\documentclass{article}
\usepackage[top=2cm, bottom=2cm, left=2cm, right=2cm]{geometry}
\usepackage[alphabetic,lite]{amsrefs}                                          % for bibliography
\usepackage{amssymb}                                                           % for math symbols
\usepackage{amsthm}                                                            % for math theorems
\usepackage{appendix}                                                          % for appendices
\usepackage{array}                                                             % for tabulation
\usepackage{bm}                                                                % for bold letters in math
\usepackage{colonequals}                                                       % for better ":="
\usepackage{color}                                                             % for colorful text
\usepackage{enumitem}                                                          % for enumerating
\usepackage{graphicx}                                                          % for inserting pictures
\usepackage{indentfirst}                                                       % for first line indent
\usepackage{mathrsfs}                                                          % for "mathscr"
\usepackage{mathtools}                                                         % for better "amsmath"
\usepackage{moreenum}                                                          % for enumerating by Greek letters
\usepackage{stmaryrd}                                                          % for (at least) maps from
\usepackage{verbatim}                                                          % for "verbatim"
\usepackage[all,cmtip]{xy}                                                     % for commutative diagrams

\definecolor{tianred}{rgb}{0.79, 0.17, 0.57}                                   % for nice red
\definecolor{tianblue}{rgb}{0.0, 0.22, 0.66}                                   % for nice blue
\definecolor{tianpink}{rgb}{0.88, 0.56, 0.59}                                  % for nice pink
\definecolor{tiangreen}{rgb}{0.24, 0.82, 0.44}                                 % for nice green
\usepackage[colorlinks,
            linkcolor=tianred,
            anchorcolor=tianblue,
            citecolor=tiangreen,
            urlcolor=tianpink]{hyperref}                                       % for references
\usepackage{cleveref}                                                          % for better references

\usepackage[OT2,T1]{fontenc}
\DeclareSymbolFont{cyrletters}{OT2}{wncyr}{m}{n}
\DeclareMathSymbol{\RBe}{\mathalpha}{cyrletters}{"42}                          % Russian B (cyrillic letter)
\DeclareMathSymbol{\Che}{\mathalpha}{cyrletters}{"51}                          % Russian Y (cyrillic letter)
\DeclareMathSymbol{\Sha}{\mathalpha}{cyrletters}{"58}                          % Russian 山 (cyrillic letter)

\makeatletter
\DeclareRobustCommand\widecheck[1]{{\mathpalette\@widecheck{#1}}}
\def\@widecheck#1#2{%
    \setbox\z@\hbox{\m@th$#1#2$}%
    \setbox\tw@\hbox{\m@th$#1%
       \widehat{%
          \vrule\@width\z@\@height\ht\z@
          \vrule\@height\z@\@width\wd\z@}$}%
    \dp\tw@-\ht\z@
    \@tempdima\ht\z@ \advance\@tempdima2\ht\tw@ \divide\@tempdima\thr@@
    \setbox\tw@\hbox{%
       \raise\@tempdima\hbox{\scalebox{1}[-1]{\lower\@tempdima\box
\tw@}}}%
    {\ooalign{\box\tw@ \cr \box\z@}}}
\makeatother

\theoremstyle{plain}      \newtheorem{thm}{Theorem}[section]                   % 定理-英
\theoremstyle{plain}                        % 定理-法
\theoremstyle{plain}      \newtheorem{lem}[thm]{Lemma}                         % 引理-英
\theoremstyle{plain}                             % 引理-法
\theoremstyle{plain}      \newtheorem{cor}[thm]{Corollary}                     % 推论-英
\theoremstyle{plain}                         % 推论-法
\theoremstyle{plain}      \newtheorem{prop}[thm]{Proposition}                  % 命题-英=法
\theoremstyle{plain}      \newtheorem{conjecture}[thm]{Conjecture}             % 猜想-英=法
\theoremstyle{definition} \newtheorem{rmk}[thm]{Remark}                        % 注释-英
\theoremstyle{definition}                      % 注释-法
\theoremstyle{definition} \newtheorem{df}[thm]{Definition}                     % 定义-英
\theoremstyle{definition}                   % 定义-法
\theoremstyle{definition} \newtheorem{eg}[thm]{Example}                        % 例子-英
\theoremstyle{definition}                        % 例子-法
\theoremstyle{definition}                        % 公理-英
\theoremstyle{definition}                      % 公理-法
\theoremstyle{definition}                    % 假设-英
\theoremstyle{definition}                  % 假设-法
\theoremstyle{definition}                        % 习题-英
\theoremstyle{definition}                       % 习题-法
\theoremstyle{definition}                    % 小结-英
\theoremstyle{definition}                  % 符号-英=法
\theoremstyle{definition}          % 构造-英=法
\theoremstyle{definition}              % 约定-英=法
\theoremstyle{definition}                  % 问题-英=法
\theoremstyle{definition} \newtheorem{prop-df}[thm]{Proposition-Definition}    % 命题-定义

\theoremstyle{definition}
\newtheorem*{construction*}{Construction}                                      % 构造-不标号
\newtheorem*{conjecture*}{Conjecture}                                          % 猜想-不标号
\newtheorem*{hypothesis*}{Hypothesis}                                          % 假设-不标号
\newtheorem*{convention*}{Convention}                                          % 约定-不标号
\newtheorem*{notation*}{Notation}                                              % 符号-不标号
\newtheorem*{summary*}{Summary}                                                % 小结-不标号
\newtheorem*{qt*}{Question}                                                    % 问题-不标号
\newtheorem*{rmk*}{Remark}                                                     % 注释-不标号
\newtheorem*{fact*}{Fact}                                                      % 事实-不标号
\newtheorem*{lizi*}{Example}                                                   % 例子-不标号
\newtheorem*{df*}{Definition}                                                  % 定义-不标号
\theoremstyle{plain}
\newtheorem*{thm*}{Theorem}                                                    % 定理-不标号

\crefname{thm}{Theorem}{Theorems}                                              %
\crefname{thme}{Th\'eo\`eme}{Th\'eo\`emes}
\crefname{lem}{Lemma}{Lemmas}
\crefname{lemme}{Lemme}{Lemmes}
\crefname{eg}{Example}{Examples}
\crefname{ege}{Exemple}{Exemples}
\crefname{rmk}{Remark}{Remarks}
\crefname{rmke}{Remarque}{Remarques}
\crefname{cor}{Corollary}{Corollaries}
\crefname{core}{Corollaire}{Corollaires}
\crefname{df}{Definition}{Definitions}
\crefname{dfe}{D\'efinition}{D\'efinitions}
\crefname{question}{Question}{Questions}
\crefname{prop}{Proposition}{Propositions}
\crefname{conjecture}{Conjecture}{Conjectures}

\newcommand{\benum}{\begin{enumerate}[label={{\upshape(\alph*)}}]}             % (a),(b),(c),etc
\newcommand{\benuma}{\begin{enumerate}[label={{\upshape(\arabic*)}}]}          % (1),(2),(3),etc
\newcommand{\benumr}{\begin{enumerate}[label={{\upshape(\roman*)}}]}           % (i),(ii),(iii),etc
\newcommand{\eenum}{\end{enumerate}}
\newcommand{\bconj}{\begin{conjecture}}
\newcommand{\econj}{\end{conjecture}}
\newcommand{\bconjnn}{\begin{conjecture*}}
\newcommand{\econjnn}{\end{conjecture*}}
\newcommand{\begs}{\begin{eg}\hfill\benuma}                                    % examples
\newcommand{\eegs}{\eenum\end{eg}}                                             % examples
\newcommand{\brmks}{\begin{rmk}\hfill\benuma}                                  % remarks
\newcommand{\ermks}{\eenum\end{rmk}}                                           % remarks
\newcommand{\bitem}{\begin{itemize}}                                           % itemize
\newcommand{\eitem}{\end{itemize}}                                             % itemize
\newcommand{\be}{\begin{equation}}                                             % equation
\newcommand{\ee}{\end{equation}}                                               % equation
\newcommand{\benn}{\begin{equation*}}                                          % non-numbering
\newcommand{\eenn}{\end{equation*}}                                            % non-numbering
\newcommand{\bqt}{\begin{qt*}\rm}                                              % question
\newcommand{\eqt}{\end{qt*}}                                                   % question
\newcommand{\bqtr}{\begin{qt*}\rm\coLR}                                        % question in red color
\newcommand{\eqtr}{\end{qt*}}                                                  % question in red color
\newcommand{\beac}{\begin{equation}\begin{array}{c}}                           % diagram central numbering
\newcommand{\eeac}{\end{array}\end{equation}}                                  % diagram central numbering
\newcommand{\beqn}{\begin{eqnarray*}}
\newcommand{\eeqn}{\end{eqnarray*}}
\newcommand{\bdf}{\begin{df}}
\newcommand{\bdfhf}{\begin{df}\hfill}
\newcommand{\edf}{\end{df}}
\newcommand{\brmk}{\begin{rmk}}
\newcommand{\brmkhf}{\begin{rmk}\hfill}
\newcommand{\ermk}{\end{rmk}}

\newcommand{\s}{\mathscr}

   \newcommand{\BD}{\mathbf{D}}

   \newcommand{\BL}{\mathbf{L}}

   \newcommand{\BR}{\mathbf{R}}

\newcommand{\CC}{\mathcal{C}}

  \newcommand{\CJ}{\mathcal{J}}
  
\newcommand{\CM}{\mathcal{M}}  
\newcommand{\CO}{\mathcal{O}}  
\newcommand{\CQ}{\mathcal{Q}}  
  \newcommand{\CT}{\mathcal{T}}

  \newcommand{\BBD}{\mathbb{D}}
  
\newcommand{\BBG}{\mathbb{G}}  \newcommand{\BBH}{\mathbb{H}}
  
\newcommand{\BBK}{\mathbb{K}}

  \newcommand{\sD}{\mathscr{D}}

  \newcommand{\rmD}{\mathrm{D}}

\newcommand{\olK}{\overline{K}}

\newcommand{\ulZ}{\underline{\Z}}

\newcommand{\A}{\mathbb{A}}                                                    % affine schemes
\newcommand{\G}{\mathbb{G}}                                                    % the group G_a, G_m
\newcommand{\Q}{\mathbb{Q}}                                                    % rational numbers
\newcommand{\Z}{\mathbb{Z}}                                                    % integral numbers
\newcommand{\QZ}{\mathbb{Q}/\mathbb{Z}}                                        % rational numbers modulo integers
\newcommand{\Qp}{\mathbb{Q}_{p}}                                               % p-adic numbers
\newcommand{\Zp}{\mathbb{Z}_{p}}                                               % p-adic integers
\newcommand{\QZp}{\Qp/\Zp}                                                     % p-adic numbers modulo integers
\newcommand{\Ql}{\mathbb{Q}_{\ell}}                                            % l-adic numbers
\newcommand{\Zl}{\mathbb{Z}_{\ell}}                                            % l-adic integers
\newcommand{\QZl}{\Ql/\Zl}                                                     % l-adic numbers modulo integers
\newcommand{\al}{\alpha}                                                       % alpha
\newcommand{\og}{\omega}                                                       % omega
\newcommand{\ra}{\rightarrow}                                                  % right arrow
\newcommand{\Ra}{\Rightarrow}                                                  % Right arrow
\newcommand{\stra}[1]{\stackrel{#1}{\ra}}

\newcommand{\mt}{\mapsto}                                                      % mapsto
\newcommand{\coLR}{\textcolor[rgb]{1.00,0,0}}                                  %亮红

\newcommand{\ol}{\overline}                                                    % overline
\newcommand{\ul}{\underline}                                                   % underline
\newcommand{\wh}{\widehat}                                                     % over widehat
\newcommand{\wc}{\widecheck}                                                   % over widecheck
\newcommand{\ce}{\colonequals}                                                 % colon + equal
\newcommand{\es}{\varnothing}                                                  % emptyset
\newcommand{\ui}{^{-1}}                                                        % upper inverse
\newcommand{\uun}{^{(1)}}                                                      % upper un
\newcommand{\lip}{\langle}                                                     % left inner product
\newcommand{\rip}{\rangle}                                                     % right inner product
\newcommand{\xm}{\xymatrix}
\newcommand{\drl}{\varinjlim}                                                  % direct limit
\newcommand{\prl}{\varprojlim}                                                 % inverse limit
\DeclareMathOperator{\Hom}{Hom}                                                % Hom(-,-)
\renewcommand{\hom}{\Hom}                                                      % Hom(-,-) again
\DeclareMathOperator{\ext}{Ext}                                                % Ext group
\DeclareMathOperator{\Ext}{\mathbf{Ext}}                                       % Hyper-tor
\newcommand{\tstsumlim}{\tst\sum\limits}                                       % text style sum limited
\newcommand{\tstprodlim}{\tst\prod\limits}                                     % text style product limited
\newcommand{\tstopslim}{\tst\bigoplus\limits}                                  % text style oplus limited
\DeclareMathOperator{\ulhom}{\ul{\hom}}                                        % underlined Hom(-,-)
\DeclareMathOperator{\ulext}{\ul{\ext}}                                        % underlined Ext(-,-)
\DeclareMathOperator{\ulpic}{\ul{\pic}}                                        % underlined Pic

\DeclareMathOperator{\Div}{Div}                                                % Divisor group
\DeclareMathOperator{\pic}{Pic}                                                % Picard group
\DeclareMathOperator{\gal}{Gal}                                                % Galois group
\newcommand{\galalg}[1]{\gal(\ol{{#1}}|{#1})}                                  % Abosolute Galois group
\DeclareMathOperator{\br}{Br}                                                  % Brauer group
\DeclareMathOperator{\Image}{Im}                                               % Image
\renewcommand{\Im}{\Image}                                                     % Image, imaginary part
\DeclareMathOperator{\Ker}{Ker}                                                % Kernel
\renewcommand{\ker}{\Ker}                                                      % Kernel
\DeclareMathOperator{\cok}{Coker}                                              % Cokernel
\DeclareMathOperator{\dif}{\partial}                                           % differential - partial

\DeclareMathOperator{\e}{Spec}                                                 % Spectrum
\newcommand{\ab}{\mathrm{ab}}                                                  % Abelianization
\newcommand{\alg}{\mathrm{alg}}                                                % Algebraic part
\newcommand{\tors}{\mathrm{tors}}                                              % Torsion part
\newcommand{\cts}{\mathrm{cont}}                                               % Continuous
\renewcommand{\ss}{{\mathrm{ss}}}                                              % semi-simple
\newcommand{\sconn}{\mathrm{sc}}                                               % simply connected
\newcommand{\torus}{\mathrm{tor}}                                              % maximal toric quotient
\newcommand{\nr}{{\mathrm{nr}}}                                                % non-ramifee cohomology
\newcommand{\sh}{\mathrm{sh}}                                                  % D^1_{sh}/ strict henselization
\newcommand{\rel}{\mathrm{rel}}                                                % relative
\DeclareMathOperator{\ops}{\oplus}                                             % Direct sum
\DeclareMathOperator{\ots}{\otimes}                                            % Tensor product
\DeclareMathOperator{\deots}{\otimes^{\BL}}                                    % Derived
\DeclareMathOperator{\dtp}{\otimes^{\BL}}                                      % Derived tensor product
\DeclareMathOperator{\pr}{pr}                                                  % projection
\newcommand{\gm}{\BBG_m}                                                       % multiplicative group
\DeclareMathOperator{\rad}{rad}                                                % Radical
\newcommand{\Gss}{G^{\mathrm{ss}}}                                             % semi-simple subgroup of G
\newcommand{\Hss}{H^{\mathrm{ss}}}                                             % semi-simple subgroup of H
\newcommand{\Bss}{B^{\mathrm{ss}}}                                             % Borel subgroup of G^ss
\newcommand{\Gsc}{G^{\mathrm{sc}}}                                             % simply connected covering of G^ss
\newcommand{\Hsc}{H^{\mathrm{sc}}}                                             % simply connected covering of H^ss
\newcommand{\Tsc}{T^{\mathrm{sc}}}                                             % maximal torus of G^sc
\newcommand{\Gtor}{G^{\mathrm{tor}}}                                           % maximal toric quotient of G
\newcommand{\Htor}{H^{\mathrm{tor}}}                                           % maximal toric quotient of H
\DeclareMathOperator{\res}{res}                                                % restriction
\newcommand{\hnr}{H_{\nr}}                                                     % unramified cohomology
\newcommand{\munots}[1]{\mu_n^{\ots{#1}}}                                      % mu_n^{\otsi}
\newcommand{\cmdm}{commutative diagram~}
\newcommand{\cmdms}{commutative diagrams~}

\newcommand{\distri}{distinguished triangle~}

\newcommand{\flasres}{flasque resolution~}

\newcommand{\finiexp}{finite exponent~}

\newcommand{\chlg}{cohomology~}

\newcommand{\hchlg}{hypercohomology~}
\newcommand{\hchlgs}{hypercohomologies~}
\newcommand{\hchlgbk}{hypercohomology}
\newcommand{\ingen}{In general,~}
\newcommand{\inpart}{In particular,~}

\newcommand{\ttiff}{if and only if~}
\newcommand{\tths}{homogeneous space~}
\newcommand{\tthss}{homogeneous spaces~}

\newcommand{\ttes}{exact sequence~}

\newcommand{\ses}{short exact sequence~}

\newcommand{\les}{long exact sequence~}

\newcommand{\wa}{weak approximation~}
\newcommand{\wabk}{weak approximation}

\newcommand{\rescores}{restriction-corestriction~}
\newcommand{\wrt}{with respect to~}

\newcommand{\TFAE}{The following are equivalent:~}
\newcommand{\ppair}{perfect pairing~}

\newcommand{\ttai}{homomorphism~}
\newcommand{\ttais}{homomorphisms~}

\newcommand{\alggr}{algebraic group~}

\newcommand{\maxtorus}{maximal torus~}

\newcommand{\ttsssc}{semi-simple simply connected~}

\newcommand{\JLCT}{Colliot-Th\'el\`ene~}
\newcommand{\JLCTbk}{Colliot-Th\'el\`ene}

\newcommand{\HSerre}{Hochschild--Serre~}
\newcommand{\Kumseq}{Kummer sequence~}

\newcommand{\ArtVer}{Artin--Verdier~}                                          % Artin--Verdier duality
\newcommand{\TateSha}{Tate--Shafarevich~}                                      % Tate--Shafarevich group

\newcommand{\qand}{\quad\ \textup{and}\quad\ }                                 % q=quatre, 4 en francais
\newcommand{\itm}{\item}
\newcommand{\tst}{\textstyle}

\newcommand{\pff}{\proof\hfill}

\begin{document}
\title{\textbf{Obstructions to weak approximation for reductive groups over $p$-adic function fields}}

\author{Yisheng TIAN}

%\keywords{$p$-adic function fields, arithmetic duality theorems, obstructions to weak approximation.}

\maketitle

\begin{abstract}
We establish arithmetic duality theorems for short complexes of tori associated to reductive groups over $p$-adic function fields. Using arithmetic dualities,
we deduce obstructions to weak approximation for certain reductive groups (especially quasi-split ones)
and relate this obstruction to an unramified Galois cohomology group.
\end{abstract}

\section*{Introduction}
This article is a subsequent work of \cite{HSSz15} on the investigation of weak approximation for a connected reductive group under certain assumptions that hold for quasi-split connected reductive groups and tori.
Let $X$ be a smooth projective and geometrically integral curve defined over some $p$-adic field $k$ and let $K$ be the function field of $X$.
Since each closed point $v\in X^{(1)}$ induces a discrete valuation of $K$ (recall that the local ring $\CO_{X,v}$ is a $1$-dimensional regular local ring, hence a discrete valuation ring), the completion $K_v$ of $K$ with respect to $v$ makes sense and hence we can ask typical questions concerning the arithmetic of algebraic $K$-groups.
For example, Harari and Szamuely studied the cohomological obstruction to the kernel of $H^1(K,G)\to\prod_{v\in X^{(1)}} H^1(K_v,G)$ being trivial in \cite{HSz16}*{Section 6} for connected linear reductive groups.
Also, Harari, Scheiderer and Szamuely described the closure of $T(K)$ in $\prod_{v\in X^{(1)}} T(K_v)$ \wrt the product of $v$-adic topologies for a $K$-torus $T$ in \cite{HSSz15}.
We want to extend the results of \cite{HSSz15} to a certain class of connected linear reductive groups and to \tthss over $K$ under such groups.

\vspace{1em}
The first main step is to establish global duality for the \TateSha group of short complexes of $K$-tori.

\begin{thm*}
Let $C=[T_1\to T_2]$ be an arbitrary complex of $K$-tori concentrated in degree $-1$ and $0$.
Let $T_1'$ and $T_2'$ be the respective dual torus of $T_1$ and $T_2$, and let $C'=[T_2'\to T_1']$.
Let $\Sha^1(C)\ce\ker\big(\BBH^1(K,C)\to\prod_{v\in X^{(1)}} \BBH^1(K_v,C)\big)$ be the \TateSha group of the complex $C$.
There is a perfect, functorial in $C$, pairing of finite groups:
$$\Sha^1(C)\times \Sha^1(C')\to \QZ.$$
\end{thm*}

The development of the global duality theorem is parallel to Izquierdo's work \cite{Diego-these} where he considered duality theorems for groups of multiplicative type over higher local fields.
Recall that a group $M$ of multiplicative type may be identified with the kernel of an epimorphism $T_1\to T_2$ of tori over the base field.
As a consequence, one may use the complex $[T_1\to T_2]$ to describe the arithmetic of $M$ and part of the dualities established by Izquierdo \cite{Diego-these} is based on the surjectivity of $T_1\to T_2$.
In our context, we get rid of the surjectivity assumption on $T_1\to T_2$ but the price is to restrict ourselves to $p$-adic function fields (which are of cohomological dimension $3$).
This refinement is important in our situation since the short complex of tori associated with a connected reductive group (see the paragraph below) is not an epimorphism in general.
Finally, we recall that just analogous to global duality results between $\Sha^1(C)$ and $\Sha^1(\wh{C})$ (where $\wh{C}=[\wh{T_2}\to \wh{T_1}]$) over number fields \cite{Dem11}*{Th\'eor\`eme 5.7}, we do not need the finiteness of $\ker(T_1\to T_2)$.

\vspace{1em}
Now we consider a connected reductive linear group $G$ over the function field $K$.
Following Deligne \cite{Deligne79}*{2.4.7} and Borovoi \cite{Bor98}, we consider the composite $\rho:\Gsc\to G^{\ss}\to G$, where $G^{\ss}=\s{D}G$ is the derived subgroup of $G$ (it is semi-simple) and $\Gsc\to G^{\ss}$ is the simply connected covering of $G^{\ss}$ (it is simply connected).
For a \maxtorus $T$ of $G$, its inverse image $\Tsc \ce \rho\ui(T)$ is a \maxtorus of $\Gsc$.
To each reductive group $G$, we associate a short complex $C=[\Tsc \to T]$ of $K$-tori concentrated in degree $-1$ and $0$.
The next result says that in general there is an obstruction to weak approximation for $G$ which is controlled by some sort of \TateSha group of $C'$.

We fix here a technical condition which is satisfied by quasi-split reductive $K$-groups (see \cref{result: quasi-split groups have *}) and $K$-tori.

\begin{df*}
We say $\Gsc$ \emph{satisfies $(*)$} if it satisfies weak approximation
and it contains a quasi-trivial maximal torus.
\end{df*}

\begin{thm*}
Let $\Sha^1_{\og}(C')$ be the group of elements in $\BBH^1(K,C')$ being trivial in $\BBH^1(K_v,C')$ for all but finitely many $v\in X^{(1)}$ and let $\Sha^1_{\og}(C')^D$ be the group of \ttais $\Sha^1_{\og}(C')\to \QZ$.
Suppose $G$ is a connected reductive group such that $\Gsc$ satisfies $(*)$.
There is an exact sequence of groups
\beac\label{sequence: OBS to WA in intro}
  1\to \ol{G(K)}\to \prod_{v\in X^{(1)}}G(K_v)\to \Sha^1_{\og}(C')^D\to \Sha^{1}(C)\to 1.
\eeac
Here $\ol{G(K)}$ denotes the closure of the diagonal image of $G(K)$ in $\prod_{v\in X^{(1)}}G(K_v)$ for the product of $v$-adic topologies.
\end{thm*}

Note that unlike the case of number field $($see Kneser \cites{Kneser62,Kneser65-SA}, Harder \cite{Harder68} and \cite{PR94}*{Proposition 7.9}$)$,
currently it is not known that semi-simple simply connected groups verify \wa over $p$-adic function fields.
However, thanks to a theorem of Th\u{a}\'{n}g \cite{Tha96}*{Theorem 1.4}, quasi-split semi-simple simply connected groups do have \wabk.

We shall see in the sequel that condition $(*)$ enables us to conclude certain maps are surjective
(see \cref{result: ab-0 surj for qs-group} and \cref{lemma: main diagram chasing}, diagram (\ref{diagram: main diagram chasing}) as well).
Let us briefly indicate other reasons why we did not manage to drop the assumption that $\Gsc$ satisfies $(*)$.

\bitem
\item
Consider a \emph{quasi-trivial} connected reductive group $H$ (its derived subgroup $\Hsc\ce \sD H$ is simply connected and the maximal toric quotient $\Htor$ is quasi-trivial) over a field $L$ of arithmetic type such that $\Hsc$ satisfies \wabk.
If $L$ is a number field, then we know $H^1(L_v,\Hsc)=1$ for non-Archimedean places $v$
and $H^1(L,\Hsc)\simeq \prod_{v|\infty}H^1(L_v,\Hsc)$ by \cite{PR94}*{Theorem 6.4 and 6.6}.
Subsequently the method of Sansuc \cite{San81} (see also the proof of \cref{lemma: main diagram chasing}) implies that $H$ satisfies \wa as well.
But if $L$ is a $p$-adic function field, we do not have the vanishing of $H^1(L_v,\Hsc)$ for all but finitely many places
and so it is not clear that $H$ satisfies \wabk.

\item
Suppose $G$ is semi-simple,
then there is an \ttes $1\to F\to \Gsc\to G\to 1$ with $F$ being a commutative finite \'etale group scheme.
In this case, Sansuc's method does not give a defect to \wa for $G$ for the same reason.
\eitem

The \ttes (\ref{sequence: OBS to WA in intro}) tells us that the group $\Sha^1_{\og}(C')^D$ can be viewed as a defect of \wa for the group $G$.
Actually, we may rephrase the \ttes (\ref{sequence: OBS to WA in intro}) in terms of the reciprocity obstruction to \wabk.
More precisely, there is a pairing which annihilates the closure of the diagonal image of $G(K)$ on the left:
\be\label{pairing: reciprocity obs pairing}
  (-,-):\tstprodlim_{v\in X^{(1)}}G(K_v)\times H_{\nr}^3(K(G),\QZ(2))\to \QZ.
\ee
See \cite{CT95SBB}*{\S4.1} for general definitions and properties of $H_{\nr}^3(K(G),\QZ(2))$ and
see \cite{HSSz15}*{pp. 18, pairing (17)} for the construction of the pairing (\ref{pairing: reciprocity obs pairing}) above.

\begin{thm*}
Let $G$ be a connected reductive group over $K$ such that $\Gsc$ satisfies $(*)$.
There exists a \ttai
$$u:\Sha^1_{\og}(C')\to H^3_{\nr}(K(G),\QZ(2))$$
such that each family $(g_v)\in\prod_{v\in X^{(1)}}G(K_v)$ satisfying $((g_v),\Im u)=0$ under the pairing $(\ref{pairing: reciprocity obs pairing})$ lies in the closure $\ol{G(K)}$ with respect to the product topology.

More precisely,
the obstruction is given by
$\Im\big(H^3(G^c,\mu_n^{\ots2})\to \hnr^3(K(G),\QZ(2))\big)$
for some sufficiently large $n$.
\end{thm*}

As some sort of complement to the present paper, we will establish global duality between $\Sha^0(C)$ and $\Sha^2(C')$, which enables one to construct a $12$-term (resp. $15$-term) Poitou--Tate \ttes of topological abelian groups associated to the complex $C=[T_1\to T_2]$ (resp. $C\dtp\Z/n$) in an upcoming paper.

\vspace{1em}
\textbf{Acknowledgements}.
This work was done under the supervision of David Harari.
I thank him for many helpful discussions and suggestions, and also for his support and patience.
I thank Jean-Louis \JLCT for valuable comments.
I appreciate the EDMH doctoral program for support and Universit\'e Paris-Sud for excellent conditions for research.

\section*{Notation and conventions}
Unless otherwise stated, all (hyper)cohomology groups will be taken with respect to the \'etale topology.
\inpart (hyper)cohomology groups over fields are identified with Galois (hyper)cohomology groups.
We fix once and for all an algebraic closure $\ol{L}$ of a field $L$ of characteristic zero.
In the sequel, \emph{varieties} over $L$ will always mean separated schemes of finite type over $L$.

\textbf{Abelian groups}.
Let $A$ be an abelian group.
We shall denote by $_nA$ (resp. $A\{\ell\}$) for the $n$-torsion subgroup (resp. $\ell$-primary subgroup with $\ell$ prime) of $A$.
Moreover, let $A_{\tors}$ be the torsion subgroup of $A$, so $A_{\tors}=\drl {_n}A$ is the direct limit of $n$-torsion subgroups of $A$.
We write $A^{\wedge}$ for the profinite completion of $A$ (that is, the inverse limit of its finite quotients),
$A_{\wedge}\ce \prl A/nA$ and
$A^{(\ell)}\ce\prl A/\ell^n$ for the $\ell$-adic completion with $\ell$ a prime number.
A torsion abelian group $A$ is of \textit{cofinite type} if $_nA$ is finite for each $n\ge 1$.
If $A$ is $\ell$-primary torsion of cofinite type, then $A/\Div A\simeq A^{(\ell)}$ where the former group is the quotient of $A$ by its maximal divisible subgroup.
For a topological abelian group $A$, we write $A^D\ce \hom_{\cts}(A,\QZ)$ for the group of continuous  homomorphisms.

\textbf{Tori}.
We write $\wh{T}$ (resp. $\wc{T}$) for the character module (resp. cocharacter module) of a $L$-torus $T$.
These are finitely generated free abelian groups endowed with a $\gal(\ol{L}|L)$-action, and moreover $\wc{T}$ is the $\Z$-linear dual of $\wh{T}$.
The \textit{dual torus} $T'$ of $T$ is the torus with character group $\wc{T}$,
that is, $\wh{T'}=\wc{T}$ as Galois modules.
We say a torus $T$ is \textit{flasque} if $H^1(L',\wc{T})=0$ for each finite extension $L'|L$ contained in $\ol{L}$.
A torus $T$ is \textit{quasi-trivial} if $\wh{T}$ admits a $\galalg{L}$-invariant $\Z$-basis.
Equivalently, $T$ is \textit{quasi-trivial} if $T\simeq\prod R_{L_i|L}\gm$
for some finite extensions $L_i|L$ contained in $\ol{L}$, where the $R_{L_i|L}$ are various Weil restrictions.

\textbf{Linear algebraic groups and homogeneous spaces}.
Let $H$ be an \alggr defined over $L$ and let $Y$ be a smooth $L$-scheme endowed with an $H$-action.
Then $Y$ is called a \emph{\tths} under $H$ if the $H(\ol{L})$-action on $Y(\ol{L})$ is transitive.
A \tths $Y$ under $H$ is a \emph{$K$-torsor} if the action is simply transitive.
By definition, \textit{reductive} algebraic groups will mean connected reductive groups.
If $H$ is reductive, then we denote $\Hss\ce \s{D}H$ for its derived subgroup (it is semi-simple).
Let $\Hsc\to \Hss$ be the universal covering of $H^{\ss}$ (it is a finite covering) with $\Hsc$ being simply connected.
Finally, we say $H$ is \emph{quasi-split} if it contains a Borel subgroup defined over the base field $L$.
Equivalently, $H$ is quasi-split if and only if some parabolic subgroup is solvable (see \cite{SGA3III}*{Section 3.9}, or \cite{Milne17}*{Chapter 17, Section I}).

\textbf{Fundamental diagram associated to a flasque resolution}.
Let $G$ be a connected reductive group over $L$.
By \cite{CT08-resolution-flasque}, there is an exact sequence of connected reductive groups:
$$1\to R\to H\to G\to 1$$
called a \emph{flasque resolution} of $G$,
where $H$ is a quasi-trivial connected reductive group which is an extension of a quasi-trivial torus by $\Gsc$,
and $R$ is a flasque $L$-torus which is central in $H$.
By \cite{CT08-resolution-flasque}*{0.3}, $H$ is an extension of the quasi-trivial torus $\Htor$ by $\Gsc$, so we obtain an identification $\Hss\simeq\Gsc$.

Recall \cite{CT08-resolution-flasque}*{pp. 94} that there is a \cmdm associated with such a flasque resolution
\beac\label{diagram: fundamental diagram associated to a flasque resolution: groups}
  \xymatrix@R=15pt{
  & 1\ar[d] & 1\ar[d] & 1\ar[d] \\
  1\ar[r] & F\ar[r]\ar[d] & \Gsc\ar[r]\ar[d] & G^{\ss}\ar[r]\ar[d] & 1\\
  1\ar[r] & R\ar[r]\ar[d] & H\ar[r]\ar[d] & G\ar[r]\ar[d] & 1\\
  1\ar[r] & M\ar[r]\ar[d] & H^{\torus}\ar[r]\ar[d] & G^{\torus}\ar[r]\ar[d] & 1\\
  & 1 & 1 & 1
  }
\eeac
with exact rows and columns.
In the bottom row of diagram (\ref{diagram: fundamental diagram associated to a flasque resolution: groups}),
$\Gtor$ and $\Htor$ are respective maximal toric quotient of $G$ and $H$,
and kernel $M$ of $H^{\torus}\to G^{\torus}$ is a group of multiplicative type.
Finally, $F$ is the kernel of $R\to H^{\torus}$.
As the kernel of $\Gsc\to G^{\ss}$, $F$ is finite and central in $\Gsc$.
It follows that $M$ is a torus, as a quotient of the torus $R$ by the finite group $F$.

Let $\rho:\Gsc\to G^{\ss}\to G$ be the composition and let $T\subset G$ be a maximal $L$-torus.
Then $\Tsc \ce \rho\ui(T)$ is a \maxtorus of $\Gsc$.
Applying \cite{CT08-resolution-flasque}*{Appendice A} to the \maxtorus $T$, we obtain a \cmdm
\beac\label{diagram: fundamental diagram associated to a flasque resolution: maixmal tori}
  \xymatrix@R=15pt{
  & 1\ar[d] & 1\ar[d] & 1\ar[d] \\
  1\ar[r] & F\ar[r]\ar[d] & \Tsc \ar[r]\ar[d] & T\cap G^{\ss}\ar[r]\ar[d] & 1\\
  1\ar[r] & R\ar[r]\ar[d] & T_H\ar[r]\ar[d] & T\ar[r]\ar[d] & 1\\
  1\ar[r] & M\ar[r]\ar[d] & H^{\torus}\ar[r]\ar[d] & G^{\torus}\ar[r]\ar[d] & 1\\
  & 1 & 1 & 1
  }
\eeac
with exact rows and columns, where $T_H\subset H$ is a \maxtorus of $H$.
Recall that $H^{\torus}$ is a quasi-trivial torus.

%\subsection{$z$-extensions}
%Again by \cite{CT08-resolution-flasque}, each reductive group $G$ admits a $z$-extension. That is, an exact sequence
%$$1\to Q\to \wt{G}\to G\to 1$$
%with $Q$ a quasi-trivial $K$-torus central in the reductive group $\wt{G}$, and the derived subgroup $\s{D}\wt{G}$ of $\wt{G}$ being the simply connected cover $\Gsc$ of the derived subgroup $G^{\ss}$ of $G$.

\textbf{Special coverings}.
An isogeny $G_0\to G$ of connected reductive $L$-groups is called a \emph{special covering}
(see \cite{San81}) if $G_0$ is the product of a \ttsssc group and a quasi-trivial torus.
For each reductive $L$-group $G$, there exist an integer $m\ge 1$ and a quasi-trivial $L$-torus $Q$ such that $G^m\times_L Q$ admits a special covering by \cite{San81}*{Lemme 1.10}.

\textbf{Motivic complexes}.
Let $Y$ be a variety over $L$.
Bloch introduced a so-called cycle complex $z^i(Y,\bullet)$ in \cite{Bloch86}.
When $Y$ is smooth, we denote the \'etale motivic complex over $Y$ by the complex of sheaves $\Z(i)\ce z^i(-,\bullet)[-2i]$ on the small \'etale site of $Y$.
For example, we have quasi-isomorphisms $\Z(0)\simeq\Z$ and $\Z(1)\simeq\G_m[-1]$ by \cite{Bloch86}*{Corollary 6.4}.
We write $A(i)\ce A\dtp \Z(i)$ for any abelian group $A$.
Finally, \cite{GL01:Bloch-Kato}*{Theorem 1.5} gives a quasi-isomorphism $\Z/n\Z(i)\simeq \mu_n^{\ots i}$ where $\mu_n$ is concentrated in degree $0$.
We shall write $\QZ(i)\ce \drl \mu_n^{\ots i}$ for the direct limit of the sheaves $\mu_n^{\ots i}$ for all $n\ge 1$.

\textbf{Function fields}. Throughout this article, $K$ will be the function field of a smooth proper and geometrically integral curve $X$ over a $p$-adic field.
For $v\in X$, we write $\CO_{X,v}$ for the local ring at $v$ and $\kappa(v)$ for its residue field.
Since $\CO_{X,v}$ is a discrete valuation ring for each $v\in X^{(1)}$,
closed points will also be referred to places in the sequel.
Moreover, $K_v$ (resp. $K_v^h$) will be the completion (resp. Henselization) of $K$ with respect to $v$ and $\CO_v$ (resp. $\CO_{X,v}^h$) will be the ring of integers in $K_v$ (resp. $K_v^h$).
Note that the fields $K$ and $K_v$ have cohomological dimension $3$.

\textbf{\TateSha groups}.
Let $C=[T_1\stra{\rho} T_2]$ be a short complex of $K$-tori concentrated in degree $-1$ and $0$ and
let $C'=[T_2'\to T_1']$ be its dual (again it is concentrated in degree $-1$ and $0$).
We put
\[
\Sha^1(C)\ce\ker\big(\BBH^1(K,C)\to \textstyle\prod\limits_{v\in X^{(1)}}\BBH^1(K_v,C)\big).
\]
For example, if $C=[0\to T]$,
then $C$ is quasi-isomorphic to $T$ and we have $\Sha^1(T)\simeq \Sha^1(C)$.
If $C=[T\to 0]$, then $C\simeq T[1]$ and hence $\Sha^2(T)\simeq \Sha^1(C)$.
For a finite set $S$ of places, we put
\[
\Sha^1_S(C)\ce\ker\big(\BBH^1(K,C)\to \tst\prod\limits_{v\notin S}\BBH^1(K_v,C)\big).
\]
We write $\Sha^1_{\og}(C)$ for locally trivial elements for all but finitely many $v$,
so $\Sha^1_{\og}(C)=\bigcup_S\Sha^1_S(C)$ with $S$ running through all finite set of places.

\section{Dualities for short complexes of tori}
Let $C=[T_1\stra{\rho} T_2]$ and let $C'=[T_2'\to T_1']$.
We fix some sufficiently small non-empty open subset $X_0$ of $X$ such that
$T_1$ and $T_2$ extends to $X_0$-tori (in the sense of \cite{SGA3II}) $\CT_1$ and $\CT_2$, respectively.
The complexes $\CC=[\CT_1\to \CT_2]$ and $\CC'=[\CT_2'\to \CT_1']$ over $X_0$ are defined analogously
(these short complexes are concentrated in degree $-1$ and $0$).
By \cite{Diego-these}*{Lemme 1.4.3} there are respective natural pairings of complexes over $K$ and $X_0$:
\[
C\ots^{\BL}C'\to \Z(2)[3]
\qand
\CC\ots^{\BL}\CC'\to \Z(2)[3].
\]
If $C=[0\to T]$ consists a single torus, then $C\dtp C'\simeq T\dtp T'[1]\to \Z(2)[3]$ is constructed in \cite{HSz16}*{pp.~4}.
Finally, we write $\BBH^i_c(X_0,\CC)\ce \BBH^i(X_0,j_{0!}\CC)$ for the compact support cohomology where $j_0:X_0\to X$ denotes the open immersion.

We begin with a list of properties of the groups under consideration in the sequel.

\begin{lem}\label{lemma: property of torsion groups I}
Let $T$ be a $K$-torus. Let $C$ and $\CC$ as above. Let $U\subset X_0$ be a non-empty open subset.
\benuma
\item The groups $\BBH^i(U,\CC\deots\Z/n)$ and $\BBH^i_c(U,\CC\deots\Z/n)$ are finite for $i\in\Z$.
\item The torsion groups $\BBH^1(U,\CC)_{\tors}$ and $\BBH^1_c(U,\CC)_{\tors}$ are of cofinite type.
\item The groups $\BBH^i(U,\CC)$ and $\BBH^i_c(U,\CC)$ are torsion of cofinite type for $i\ge 2$.
\item The group $H^1(K_v,T)$ is finite and the groups $\Sha^i(T)$ are finite for $i=1,2$.
\eenum
\end{lem}
\proof
\benuma
\item
This is proved in \cite{Diego-these}*{Proposition 1.4.4}.
\item
The first statement is a consequence of the sequence
\[0\to \BBH^i(U,\CC)/n\to \BBH^i(U,\CC\dtp \Z/n)\to {_n}\BBH^{i+1}(U,\CC)\to 0\]
induced by the \distri $\CC\to \CC\to\CC\dtp \Z/n\to \CC[1]$.
The same argument works for $\BBH^1_c(U,\CC)_{\tors}$.
\item
By \cite{HSz16}*{Corollary 3.3 and Proposition 3.4(1)}, the groups $H^i(U,\CT_2)$ and $H^{i+1}(U,\CT_1)$ are torsion of cofinite type for $i\ge 2$.
Now we deduce that $\BBH^i(U,\CC)$ is torsion by the exactness of $H^i(U,\CT_2)\to \BBH^i(U,\CC)\to H^{i+1}(U,\CT_1)$.
The group $\BBH^i(U,\CC)$ is of cofinite type thanks to the \ses $0\to \BBH^{i-1}(U,\CC)/n\to \BBH^{i-1}(U,\CC\dtp \Z/n)\to {_n}\BBH^{i}(U,\CC)\to 0$.
The same argument works for $\BBH^i_c(U,\CC)$.
\item
The group $H^1(K_v,T)$ is finite because it has \finiexp and is of cofinite type (see \cite{HSz16}*{Proposition 2.2} for more details).
The second statement is part of \cite{HSz16}*{Proposition 3.4(2)}.\qed
\eenum

\subsection{An \ArtVer style duality}
The following result is some sort of variation of the classical \ArtVer duality theorem,
which provides a more precise statement concerning the $\ell$-primary part.

\begin{prop}\label{duality: AV complex of tori degree 1 1}
Let $U\subset X_0$ be any non-empty open subset.
There is a pairing with divisible left kernel for each prime number $\ell$,
\begin{equation*}
  \BBH^i(U,\CC)\{\ell\}\times\BBH^{2-i}_c(U,\CC')^{(\ell)}\to \QZ.
\end{equation*}
\end{prop}
\begin{proof}
First, recall \cite{Diego-these}*{Proposition 1.4.4} that there is a \ppair of finite groups
\be\label{duality: AV cpx finite level}
  \BBH^i(U,\CC\ots^{\BL}\Z/n)\times \BBH^{1-i}_c(U,\CC'\ots^{\BL}\Z/n)\to \QZ
\ee
for $i\in \Z$.
The pairing $\CC\deots\CC'\to \Z(2)[3]$ induces a pairing $\BBH^i(U,\CC)\times\BBH^{2-i}_c(U,\CC')\to \QZ$ by \cite{HSz16}*{Lemma 1.1}.
In particular, we obtain pairings
\[
{_{\ell^n}}\BBH^i(U,\CC)\times\BBH^{2-i}_c(U,\CC')/\ell^n\to \QZ
\qand
\BBH^i(U,\CC)/\ell^n\times{_{\ell^n}}\BBH^{2-i}(U,\CC')\to \QZ
\]
which fit into the following \cmdm with exact rows
(it commutes by functoriality of the cup-product analogous to \cite{CTH16}*{Lemme 3.1}):
\begin{equation*}
  \xm{
  0\ar[r] & \BBH^{i-1}(U,\CC)/\ell^n\ar[r]\ar[d] & \BBH^{i-1}(U,\CC\deots\Z/\ell^n)\ar[r]\ar[d] & {_{\ell^n}}\BBH^{i}(U,\CC)\ar[r]\ar[d] & 0\\
  0\ar[r] & \big({_{\ell^n}}\BBH^{3-i}_c(U,\CC')\big)^D\ar[r] & \BBH^{2-i}_c(U,\CC'\deots\Z/\ell^n)^D\ar[r] & \big(\BBH^{2-i}_c(U,\CC')/\ell^n\big)^D\ar[r] & 0.
  }
\end{equation*}
Since the middle vertical arrow is an isomorphism by the pairing (\ref{duality: AV cpx finite level}),
we obtain an isomorphism by snake lemma
\[%beac\label{equation: local duality snake lemma I}
\BBK_n(U)
\simeq
\cok\big(\BBH^{i-1}(U,\CC)/\ell^n\to \big({_{\ell^n}}\BBH^{3-i}_c(U,\CC')\big)^D\big),
\]%eeac
where $\BBK_n(U)=\ker\big({_{\ell^n}}\BBH^{i}(U,\CC)\to (\BBH^{2-i}_c(U,\CC')/\ell^n)^D\big)$.
Taking direct limit over all $n$ yields an isomorphism
\[
\drl_n\BBK_n(U)
\simeq
\drl_n\cok\big(\BBH^i(U,\CC)/\ell^n\to \big({_{\ell^n}}\BBH^{2-i}_c(U,\CC')\big)^D\big).
\]
The latter limit is a quotient of the divisible group
$\drl \big({_{\ell^n}}\BBH^{2-i}_c(U,\CC')\big)^D
\simeq \big(\prl{_{\ell^n}}\BBH^{2-i}_c(U,\CC')\big)^D$,
so it is also divisible.
Indeed, since $\BBH^{2-i}_c(U,\CC')\{\ell\}$ is a torsion group of cofinite type,
so it is of the form $\big(\QZl\big)^{\ots r}\ops F_{\ell}$ where $F_{\ell}$ is a finite $\ell$-group.
Thus the dual of its Tate module $\prl{_{\ell^n}}\BBH^{2-i}_c(U,\CC')$ is a direct sum of copies of $\Q_{\ell}/\Z_{\ell}$,
i.e. $\big(\prl{_{\ell^n}}\BBH^{2-i}_c(U,\CC')\big)^D$ is divisible.
% see https://mathoverflow.net/questions/293388/co-finite-type-abelian-groups
Being isomorphic to a quotient of the divisible group $\big(\prl{_{\ell^n}}\BBH^{2-i}_c(U,\CC')\big)^D$,
we see that $\drl_n\BBK_n(U)$ is divisible as well.
Passing to the direct limit over all $n$ yield an exact sequence (by definition of $\BBK_n(U)$ and exactness of direct limit) of abelian groups
\[
0\to \drl_n\BBK_n(U)\to \BBH^i(U,\CC)\{\ell\}\to \big(\BBH^{2-i}_c(U,\CC')^{(\ell)}\big)^D
\]
which guarantees the required pairing having divisible left kernel.
\end{proof}

\begin{rmk}
We shall see later in \cref{result: global duality degree 1 and 1} that
there exists a non-empty open subset $U_0$ of $X_0$ such that the induced map $\BBH^1(U,\CC)\{\ell\}\to\big(\BBH^1_c(U,\CC')^{(\ell)}\big)^D$ is an isomorphism for each $U\subset U_0$,
because the direct limit $\drl_n\BBK_n(U)$ is contained in a finite group (so it vanishes as it a finite divisible group).
\end{rmk}

\subsection{Local dualities for short complex of tori}
In this subsection, we prove local dualities for the completion $K_v$ and the Henselization $K_v^h$ with respect to $v$.

\begin{prop}[Local dualities]\label{duality: local cpx of tori degree 0 1}
Let $\ell$ be a prime number.
\benuma
\item There is a \ppair functorial in $C$ between discrete and profinite groups:
\[\BBH^{1}(K_v,C)\times \BBH^{0}(K_v,C')^{\wedge}\to \QZ.\]

\item There is a \ppair functorial in $C$ between finite groups:
\begin{equation*}
{_{\ell^n}}\BBH^1(K_v,C)\times \BBH^0(K_v,C')/\ell^n\to \QZ.
\end{equation*}
\eenum
\end{prop}
\proof\hfill
\benuma
\item The \distri $T_2'\to T_1'\to C'\to T_2'[1]$ induces an exact sequence
$$H^0(K_v,T_2')\to H^0(K_v,T_1')\to \BBH^0(K_v,C')\to H^1(K_v,T_2')\to H^1(K_v,T_1').$$
Since $H^1(K_v,T_2')$ is finite by \cref{lemma: property of torsion groups I}(4),
by \cite{HSz05}*{Appendix, Proposition} there is an exact sequence
\[
H^0(K_v,T_1')^{\wedge}\to \BBH^0(K_v,C')^{\wedge}\to H^1(K_v,T_2')\to H^1(K_v,T_1')
\]
and a complex $$H^0(K_v,T_2')^{\wedge}\to H^0(K_v,T_1')^{\wedge}\to \BBH^0(K_v,C')^{\wedge}.$$
Now the statement follows by exactly the same argument as \cite{Diego-these}*{Proposition 1.4.9(ii)}.

\item Consider the following exact \cmdm
\begin{equation*}
  \xymatrix{
  0\ar[r] & \BBH^0(K_v,C)/\ell^n\ar[r]\ar[d] & \BBH^0(K_v,C\deots\Z/\ell^n)\ar[r]\ar[d] & {_{\ell^n}}\BBH^1(K_v,C)\ar[r]\ar[d] & 0\\
  0\ar[r] & \big({_{\ell^n}}\BBH^1(K_v,C')\big)^D\ar[r] & \BBH^0(K_v,C'\deots\Z/\ell^n)^D\ar[r] & \big(\BBH^0(K_v,C')/\ell^n\big)^D\ar[r] & 0
  }
\end{equation*}
with the middle vertical arrow being an isomorphism of finite groups (see the proof of \cite{Diego-these}*{Proposition 1.4.9(i)}).
It follows that the right vertical arrow is surjective.
Moreover, we have a \cmdm
\begin{equation*}
  \xymatrix{
  {_{\ell^n}}\BBH^1(K_v,C)\ar[r]\ar[d] & \big(\BBH^0(K_v,C')/\ell^n\big)^D\ar[d]\\
  \BBH^1(K_v,C)\ar[r] & \big(\BBH^0(K_v,C')^{\wedge}\big)^D
  }
\end{equation*}
where the lower horizontal arrow is injective by (1).
Therefore the upper horizontal arrow is also injective and hence it is an isomorphism.
\qed
\eenum

\begin{rmk}
By a similar argument as \cref{duality: local cpx of tori degree 0 1}(2),
we obtain a \ppair between profinite and discrete groups
\[
\BBH^{0}(K_v,C)_{\wedge}\times \BBH^{1}(K_v,C')\to \QZ.
\]
So %there are isomorphisms
%\[
%\drl_n {_{n}}\BBH^1(K_v,C)\simeq \big(\prl_n\BBH^0(K_v,C')/n\big)^D=\big(\BBH^0(K_v,C')_{\wedge}\big)^D,
%\]
%i.e.
we can identify $\BBH^0(K_v,C)_{\wedge}$ with the profinite completion $\BBH^0(K_v,C)^{\wedge}$ by \cref{duality: local cpx of tori degree 0 1}(1).
\end{rmk}

\begin{cor}\label{duality: local cpx of tori limits 0 1}
Let $\ell$ be a prime number.
\benuma
\item
There is a \ppair between discrete and profinite groups:
\[\BBH^1(K_v,C)\{\ell\}\times\BBH^0(K_v,C')^{(\ell)}\to \QZ.\]
\item
There is a \ppair between locally compact groups:
\[
\big(\tst\prod\limits_{v\in X^{(1)}}\BBH^1(K_v,C)\big)\{\ell\}
\times
\big(\tst\bigoplus\limits_{v\in X^{(1)}}\BBH^0(K_v,C')\big)^{(\ell)}
\to \QZ.
\]
More precisely, the former group is a direct limit of profinite groups and the latter is a projective limit of discrete torsion groups.
\eenum
\end{cor}
\proof
We apply the local duality \cref{duality: local cpx of tori degree 0 1}(2), i.e. the isomorphism ${_{\ell^n}}\BBH^1(K_v,C)\simeq \big(\BBH^0(K_v,C')/\ell^n\big)^D$.\hfill
\benuma
\item
Passing to the direct limit over all $n$ yields the isomorphism $\BBH^1(K_v,C)\{\ell\}\simeq \big(\BBH^0(K_v,C')^{(\ell)}\big)^D$.

\item
Taking product over all places gives isomorphisms
\[{_{\ell^n}}\big(\textstyle\prod\limits_v\BBH^1(K_v,C)\big)
\simeq\textstyle\prod\limits_v\big(\BBH^0(K_v,C')/\ell^n\big)^D
\simeq \big(\big(\textstyle\bigoplus\limits_v\BBH^0(K_v,C')\big)/\ell^n\big)^D.\]
Thus the desired \ppair follows by passing to the direct limit over all $n\ge 1$.\qed
\eenum

\begin{rmk}
Analogously, there is a \ppair between locally compact groups
\[
\big(\tstprodlim_{v\in X^{(1)}}\BBH^0(K_v,C)\big)\{\ell\}
\times
\big(\tstopslim_{v\in X^{(1)}}\BBH^1(K_v,C')\big)^{(\ell)}
\to \QZ.
\]
More precisely, the former group is a direct limit of profinite groups and the latter is a projective limit of discrete torsion groups.
\end{rmk}

The next lemma is probably well-known:

\begin{lem}\label{lemma: on exactness of l-adic completion}
Let $A_1\to A_2\to A_3 \to 0$ be an \ttes of abelian groups.
If ${_{\ell}}A_3$ is finite, then $A_1^{(\ell)}\to A_2^{(\ell)}\to A_3^{(\ell)}\to 0$ is exact for each prime number $\ell$.
\end{lem}
\begin{proof}
Let's say $f:A_1\to A_2$, $g:A_2\to A_3$ and $g_n:A_2/\ell^n\to A_3/\ell^n$.
Thus there is a \ses $0\to \ker g_n\to A_2/\ell^n\to A_3/\ell^n\to 0$.
Since $\ker g_n$ is a quotient of $A_1/\ell^n$, $\{\ker g_n\}$ forms a surjective system in the sense of \cite{AM69}*{Proposition 10.2} and
it follows that $0\to \prl \ker g_n\to A_2^{(\ell)}\to A_3^{(\ell)}\to 0$ is exact.
By the snake lemma, there is an \ttes
$0\to {_{\ell^n}}\ker g\to {_{\ell^n}}A_2\to {_{\ell^n}}A_3\to (\ker g)/\ell^n\to \ker g_n\to 0$.
But ${_{\ell^n}}A_3$ is finite by assumption,
we conclude that $(\ker g)^{(\ell)}\to \prl \ker g_n$ is surjective by Mittag--Leffler condition.
Finally, let $\ker f_n\ce \ker\big(A_1/\ell^n\to (\ker g)/\ell^n\big)$.
Then $\{\ker f_n\}$ is a surjective system (because $\ker f/\ell^n\to \ker f_n$ is surjective),
and hence $A_1^{(\ell)}\to (\ker g)^{(\ell)}$ is surjective.
Summing up, the sequence $A_1^{(\ell)}\to A_2^{(\ell)}\to A_3^{(\ell)}\to 0$ is exact.
\end{proof}

\begin{lem}\label{lemma: Henselization vs completion}
Let $\ell$ be a prime number. Let $T$ be a $K$-torus and let $C=[T_1\to T_2]$ be as above. \hfill
\benuma
\item The natural map $H^0(K_v^h,T)\to H^0(K_v,T)$ induces
an isomorphism %$H^0(K_v^h,T)^{\wedge}\to H^0(K_v,T)^{\wedge}$ and
$H^0(K_v^h,T)^{(\ell)}\simeq H^0(K_v,T)^{(\ell)}$.
Moreover, there is an isomorphism %$\BBH^0(K_v^h,C)^{\wedge}\simeq\BBH^0(K_v,C)^{\wedge}$ and
$\BBH^0(K_v^h,C)^{(\ell)}\simeq\BBH^0(K_v,C)^{(\ell)}$.

\item For $i\ge 1$, there is an isomorphism $\BBH^i(K_v^h,C)\to \BBH^i(K_v,C)$.
\eenum
\end{lem}
\proof\hfill
\benuma
\item The same argument as \cite{Dem11}*{Lemme 3.7} yields an isomorphism
$H^0(K_v^h,T)/\ell^n\simeq H^0(K_v,T)/\ell^n$.
Therefore the first assertion follows by passing to the inverse limit over all $n$.
For the second statement, since $H^1(K_v^h,T_i)\simeq H^1(K_v,T_i)$ is finite for $i=1,2$,
there is a \cmdm of complexes with rows exact at the last four terms by \cref{lemma: on exactness of l-adic completion}:
\begin{displaymath}
  \xymatrix{
  H^0(K_v^h,T_1)^{(\ell)}\ar[r]\ar[d] & H^0(K_v^h,T_2)^{(\ell)}\ar[r]\ar[d] & \BBH^0(K_v^h,C)^{(\ell)}\ar[r]\ar[d] & H^1(K_v^h,T_1)\ar[r]\ar[d] & H^1(K_v^h,T_2)\ar[d]\\
  H^0(K_v,T_1)^{(\ell)}\ar[r] & H^0(K_v,T_2)^{(\ell)}\ar[r] & \BBH^0(K_v,C)^{(\ell)}\ar[r] & H^1(K_v,T_1)\ar[r] & H^1(K_v,T_2).
  }
\end{displaymath}
Now all the vertical arrows except the middle one are isomorphisms, and hence the middle one is also an isomorphism by the $5$-lemma.

\item
By \cite{HSz16}*{Corollary 3.2}, we know that $H^i(K_v^h,T)\simeq H^i(K_v,T)$ for each $i\ge 1$ and for each $K$-torus $T$.
Thus the isomorphism $\BBH^i(K_v^h,C)\simeq \BBH^i(K_v,C)$ for each $i\ge 1$ follows after applying d\'evissage to the \distri $T_1\to T_2\to C\to T_1[1]$.
\qed
\eenum

\begin{cor}\label{duality: local Hensel degree 0 1}
There is a \ppair between direct limit of profinite groups and projective limit of discrete torsion groups:
\[\big(\tstprodlim_{v\in X^{(1)}}\BBH^1(K_v,C)\big)\{\ell\}\times \big(\tstopslim_{v\in X^{(1)}}\BBH^0(K_v^h,C')\big)^{(\ell)}\to \QZ.\]
\end{cor}
\begin{proof}
The same argument as \cref{duality: local cpx of tori degree 0 1} yields a \ppair ${_{\ell^n}}\BBH^1(K_v^h,C)\times\BBH^0(K_v^h,C')/\ell^n\to\QZ$ of finite groups.
Therefore \cref{lemma: Henselization vs completion}(2) implies that $\BBH^0(K_v^h,C')/\ell^n\simeq \BBH^0(K_v,C')/\ell^n$.
The desired \ppair is an immediate consequence by the same argument as \cref{duality: local cpx of tori limits 0 1}(2).
\end{proof}

\subsection{Global dualities for short complex of tori}
The goal of this subsection is to establish a \ppair $\Sha^1(C)\times \Sha^1(C')\to \QZ$ between finite groups.
We first prove the finiteness of $\Sha^1(C)$ and $\Sha^1(C')$.

\begin{lem}
Let $C=[T_1\to T_2]$ be a short complex of tori.
The group $\Sha^1_{\og}(C)$ is of finite exponent.
\end{lem}
\proof
Let $L|K$ be a finite Galois extension that splits both $T_1$ and $T_2$.
Then for $i=1,2$, the $L$-tori $T_{i,L}=T_i\times_KL$ are products of $\gm$,
and $H^1(L,T_{2,L})=0$ by Hilbert's theorem $90$.
The \distri $T_1\to T_2\to C\to T_1[1]$ induces a \cmdm
\begin{displaymath}
  \xymatrix{
  & \BBH^1(K,C)\ar[r]^-{\dif}\ar[d]_{\res} & H^2(K,T_1)\ar[d]^{\res}\\
  0 \ar[r] & \BBH^1(L,C_L)\ar[r] & H^2(L,T_{1,L})
  }
\end{displaymath}
where $C_L=[T_{1,L}\to T_{2,L}]$.
Recall \cite{HSSz15}*{Lemma 3.2(2)} that $\Sha^2_{\og}(T_1)$ is of finite exponent,
and hence a \rescores argument shows that $\Sha^1_{\og}(C)$ is of finite exponent.
\qed

The following statement on compact support \hchlg is the analogue of \cite{HSz16}*{Proposition 3.1}.

\begin{prop}\label{result: l-e-s cpt-supp and 3 arrows lemma}
Let $U\subset X_0$ be a non-empty open subset.
\benuma
\item Let $V\subset U$ be a further non-empty open subset. There is an exact sequence
$$\cdots\to \BBH^i_c(V,\CC)\to \BBH^i_c(U,\CC)\to \tstopslim_{v\in U\setminus V}\BBH^i(\kappa(v),i_v^*\CC)\to \BBH^{i+1}_c(V,\CC)\to\cdots$$
where $i_v:\e\kappa(v)\to U$ is the closed immersion.

\item There is an exact sequence of \hchlg groups
$$\cdots\to \BBH^i_c(U,\CC)\to \BBH^i(U,\CC)\to \tstopslim_{v\notin U}\BBH^i(K_v^h,C)\to \BBH^{i+1}_c(U,\CC)\to \cdots$$
where $K_v^h$ is the Henselization of $K$ with respect to the place $v$ and by abuse of notation we write $C$ for the pull-back of $\CC$ by the natural morphism $\e K_v^h\to U$.

\item There is an exact sequence for $i\ge 1$:
$$\cdots\to \BBH^i_c(U,\CC)\to \BBH^i(U,\CC)\to \tstopslim_{v\notin U}\BBH^i(K_v,C)\to \BBH^{i+1}_c(U,\CC)\to \cdots$$

\item \textup{(Three Arrows Lemma).} There is a \cmdm
\[\xm{
\BBH^i_c(V,\CC)\ar[r]\ar[d] & \BBH^i_c(U,\CC)\ar[d]\\
\BBH^i(V,\CC) & \BBH^i(U,\CC).\ar[l]
}\]
\eenum
\end{prop}
\proof Actually the proofs follow from \cite{HSz16}*{Proposition 3.1} after replacing \chlg by \hchlgbk.
\benuma
\item
Applying \cite{MilneEC}*{III.~Remark 1.30} to the open immersion $V\to U$ and the closed immersion $U\setminus V\to U$ yields the required \les for \hchlgbk.
\item
The \les for \hchlg associated to the open immersion $j:U\to X$ reads as
\[\cdots\to \BBH^i_{X\setminus U}(X,j_!\CC)\to \BBH^i(X,j_!\CC)\to \BBH^i(U,\CC)\to \BBH^{i+1}_{X\setminus U}(X,j_!\CC)\to \cdots\]
Now the same argument as \cite{MilneADT}*{II. Lemma 2.4} implies the required long exact sequence.
\item
Applying \cref{lemma: Henselization vs completion}(2) to the previous \les yields the desired long exact sequence.
\item
There is an isomorphism of \hchlgs
$\BBH^i_c(U,\CC)=\BBH^i(X,j_{!}\CC)\simeq \Ext^i_X(\Z,j_{!}\CC)$ by \cite{MilneEC}*{III. Remark 1.6(e)} (where $j:U\to X$ denotes the open immersion),
thus the same argument as \cite{HSz16}*{Proposition 3.1(3)} completes the proof.
\qed
\eenum

The following result provides both the finiteness of $\Sha^1(C)$ and a crucial point in the proof of global duality.

\begin{lem}\label{result: D-2(Gm) finite exponent}
Let $U\subset X_0$ be a non-empty open subset,
and put $\rmD^2_K(U,\CT)=\Im\big(H^2_c(U,\CT)\to H^2(K,T)\big)$ for any $U$-torus $\CT$.
Then there exists a non-empty open subset $U_2\subset X_0$ such that $\rmD^2_K(U,\CT)$ is of finite exponent
for any non-empty open subset $U\subset U_2$.
\end{lem}
\begin{proof}
By a \rescores argument, it will be sufficient to show that $\rmD^2_K(U,\BBG_{m})$ is of finite exponent.
By \cite{Gro68BrauerIII}*{pp.~96, (2.9)}, there is an \ttes
\[H^0(k,\ulpic_{X/k})\to \br k\to \br X\to H^1(k,\ulpic_{X/k})\to H^3(k,\gm).\]
Since $k$ is a $p$-adic local field, we conclude $\br X/\br_0X\simeq H^1(k,\ulpic_{X/k})$ for cohomological dimension reasons.
Recall that there is a canonical \ses
$0\to \ulpic^{\circ}_{X/k}\to \ulpic_{X/k}\to \ulZ\to 0$,
thus $H^1(k,\ulpic_{X/k})$ is a quotient of $H^1(k,\ulpic^{\circ}_{X/k})$ where $\ulpic^{\circ}_{X/k}$ is an abelian variety.
But $H^1(k,\ulpic^{\circ}_{X/k})$ is dual to $\ulpic^{\circ}_{X/k}(k)$\footnote{Note that $\ulpic^{\circ}_{X/k}$ is the Jacobian of a curve, so it is isomorphic to its dual.} by Tate duality over local fields
\cite{MilneADT}*{Chapter I, Corollary 3.4},
we deduce that $H^1(k,\ulpic^{\circ}_{X/k})\simeq F_0\bigoplus (\QZp)^{\oplus r}$ with $F_0$ a finite abelian group by Mattuck's theorem (see \cite{Mattuck55} and \cite{MilneADT}*{pp.~41}).

Suppose first there is a rational point $e\in X(k)$ on $X$.
Let $e^*:\br X\to \br k\simeq\QZ$ be the induced map and put $\br_e X=\{\al\in\br X\ |\ e^*(\al)=0\}$.
Note that in this case the map $\br k\to \br X$ induced by the structural morphism $X\to \e k$ is injective and
there is an isomorphism $H^1(k,\ulpic^{\circ}_{X/k})\simeq H^1(k,\ulpic_{X/k})$.
Thus there is a split \ses $0\to \br_eX\to \br X\to \br k\to 0$ and consequently $\br_eX\simeq\br X/\br k$.
Moreover, if $e\notin U$, then $\rmD^2_K(U,\gm)\subset\ker\big(\br U\to \bigoplus_{v\notin U}\br K_v\big)\subset\br_e X$.
It follows that there is an injective map $\rmD^2_K(U,\gm)\to F_0\bigoplus (\QZp)^{\oplus r}$.
Next, we show that there is a non-empty open subset $U_2\subset X_0$ such that the decreasing sequence $\{\rmD^2_K(U,\gm)\{\ell\}\}$ is stable for $U\subset U_2$.
%of $\ell$-primary torsion groups of cofinite type
By \cite{HSz16}*{Proposition 3.4}, the group $H^2_c(U,\gm)$ is of cofinite type and hence so is $\rmD^2_K(U,\gm)$.
Since $F_0$ is finite, there exists \textit{only} finitely many $\ell\ne p$ such that $\ell$ divides the order of $F_0$.
As a consequence, there exists a non-empty open subset $U_1\subset X_0$ (which is independent of $\ell$) such that $\rmD^2_K(U,\gm)\{\ell\}=\rmD^2_K(U_1,\gm)\{\ell\}$ holds for any non-empty open subset $U\subset U_1$ and for each $\ell\ne p$ by \cite{HSz16}*{Lemma 3.7}.
Again the decreasing sequence $\{\rmD^2_K(U,\gm)\{p\}\}_{U\subset U_1}$ stabilizes,
so there exists some $U_2\subset U_1$ such that $\rmD^2_K(U,\gm)=\rmD^2_K(U_2,\gm)$ for all $U\subset U_2$.
Note that $\Sha^2(\gm)$ is the direct limit of $\rmD^2(U,\gm)$.
Letting $U$ run through all non-empty open subsets of $U_2$ yields $\rmD^2_K(U_2,\gm)=\Sha^2(\gm)=0$, where the vanishing of $\Sha^2(\gm)$ is a consequence of \cite{HSSz15}*{Lemma 3.2}.

\ingen there exists a finite Galois extension $k'|k$ such that $X(k')\ne\es$.
Put $X_{k'}\ce X\times_kk'$ and $U_{k'}\ce U\times_kk'$.
Let $K'$ be the function field of $X_{k'}$.
For any sufficiently small non-empty open subset $U$ of $X$,
we know that $\rmD^2_{K'}(U_{k'},\gm)\subset \br K'$ vanishes by the previous paragraph.
Therefore a \rescores argument implies that $\rmD^2_K(U,\gm)\subset \br K$ has finite exponent.
\end{proof}

\begin{prop}\label{result: D-1 and Sha-1 finite}
We put $\BBD_K^1(U,\CC)\ce\Im\big(\BBH^1_c(U,\CC)\to \BBH^1(K,C)\big)$.
Then there exists a non-empty open subset $U_0$ of $X_0$ such that
\be\label{equation: relation between D-one and Sha}
  \BBD_K^1(U,\CC)=\BBD_K^1(U_0,\CC)=\Sha^1(C).
\ee
for each non-empty open subset $U\subset U_0$.
Moreover, the group $\Sha^1(C)$ is finite.
\end{prop}
\begin{proof}
By \cref{result: D-2(Gm) finite exponent}, the group $\rmD^2_K(U,\CT_1)$ is of finite exponent for $U$ sufficiently small.
Since $H^1(K,T_2)$ is of finite exponent, it follows that $\BBD^1_K(U,\CC)$ is of \finiexp (say $N$) by d\'evissage.
\inpart the epimorphism $\BBH^1_c(U,\CC)\to \BBD^1_K(U,\CC)$ factors through $\BBH^1_c(U,\CC)\to \BBH^1_c(U,\CC)/N$.
Recall that $\BBH^1_c(U,\CC)/N$ is a subgroup of the finite group $\BBH^1_c(U,\CC\deots\Z/N)$, hence its quotient $\BBD^1_K(U,\CC)$ is finite.

For non-empty open subsets $V\subset U\subset X_0$ of $X_0$,
we have $\BBD_K^1(V,\CC)\subset \BBD_K^1(U,\CC)$ by covariant functoriality of $\BBH^1_c(-,\CC)$.
The decreasing sequence $\{\BBD_K^1(U,\CC)\}_{U\subset X_0}$ of finite abelian groups must be stable, hence there exists a non-empty open subset $U_0$ of $X_0$ such that $\BBD_K^1(U,\CC)=\BBD_K^1(U_0,\CC)$ for each non-empty open subset $U\subset U_0$.
Note that $\BBD^1_K(U,\CC)\subset \ker\big(\BBH^1(K,C)\to \prod_{v\notin U}\BBH^1(K_v,C)\big)$ by \cref{result: l-e-s cpt-supp and 3 arrows lemma}(3).
Letting $U$ run through all non-empty open subset of $U_0$ implies that $\BBD_K^1(U,\CC)=\BBD_K^1(U_0,\CC)=\Sha^1(C)$.
Since the former two groups are finite, so is $\Sha^1(C)$.
\end{proof}

Applying \cref{result: D-1 and Sha-1 finite} to both $C$ and $C'$, we obtain a non-empty open subset $U_0$ of $X$ such that (\ref{equation: relation between D-one and Sha}) holds for both $C$ and $C'$.
In the sequel, we fix such a non-empty open subset $U_0$ of $X$.

We will need an auxiliary Grothendieck--Serre conjecture style result (for example, see \cite{CSan87a}*{Theorem 4.1}).
Let $\CO$ be a Henselian regular local integral domain with fraction field $\Omega$ and residue field $\kappa$.
Let $\CJ=[\CT_1\stra{\rho} \CT_2]$ be a complex of $\CO$-tori concentrated in degree $-1$ and $0$.
Let $T_i=\CT_i\times_{\CO}\Omega$ be the generic fibre of $\CT_i$ for $i=1,2$ and let $J=[T_1\stra{\rho} T_2]$ be the associated complex in degree $-1$ and $0$.

\begin{prop}\label{result: GS conj for complex of tori}
The natural \ttai $\BBH^1(\CO,\CJ)\to \BBH^1(\Omega,J)$ induced by $\CO\subset\Omega$ is injective.
\end{prop}
\begin{proof}
Let $q:\CQ\to \CT_2$ be an epimorphism of $\CO$-tori with $\CQ$ being quasi-trivial (for example, we may take a flasque resolution of $\CT_2$, see \cite{CSan87a}*{(1.3.3)}).
%Let $\CQ\times_{\CT_2}\CT_1\to \CT_2$ be the map $(r,t_1)\mt q(r)\rho(t_1)\ui$ and let $\CM$ be its kernel.
%Note that $\CM$ is a group of multiplicative type.
%Let $\pr_1:\CM\to \CQ\times_{\CT_2}\CT_1\to \CQ$ and $\pr_2:\CM\to \CQ\times_{\CT_2}\CT_1\to \CT_1$ be the respective canonical projections.
Let $\CM\ce \CQ\times_{\CT_2}\CT_1$ and
let $\CQ\times_{\CT_2}\CT_1\to \CT_2$ be the map $(r,t_1)\mt q(r)\rho(t_1)\ui$.
Let $\pr_1:\CM\to \CQ$ and $\pr_2:\CM\to \CT_1$ be the respective canonical projections.
By construction of $\CM$, we have $q\circ\pr_1=\rho\circ\pr_2$.
A direct verification yields isomorphisms $\ker\pr_1\simeq\ker\rho$ and $\cok\pr_1\simeq\cok\rho$, that is,
$\CJ_0=[\CM\to \CQ]$ is quasi-isomorphic to $\CJ=[\CT_1\to \CT_2]$.
Note that $Q=\CQ\times_{\CO}\Omega$ is a quasi-trivial $\Omega$-torus and being faithfully flat is stable under base change, the same argument as above yields that $J_0=[M\to Q]$ with $M=\CM\ots_{\CO}\Omega$ is quasi-isomorphic to the complex $J=[T_1\to T_2]$.
Thus it suffices to show $\BBH^1(\CO,\CJ_0)\to \BBH^1(\Omega,J_0)$ is injective.

By construction $\CQ$ is a quasi-trivial $\CO$-torus,
thus $H^1(\CO,\CQ)=H^1(\kappa,\CQ)=0$ and $H^1(\Omega,Q)=0$ by Shapiro's lemma and Hilbert's theorem $90$.
Consequently it will be sufficient to show $H^2(\CO,\CM)\to H^2(\Omega,M)$ is injective by applying d\'evissage to the \distri $\CM\to \CQ\to \CJ_0\to \CM[1]$.
Take an exact sequence $1\to \CM\to \CQ_1\to \CQ_2\to 1$ over $\CO$ with $\CQ_1$ a quasi-trivial $\CO$-torus and $\CQ_2$ an $\CO$-torus (for example, see \cite{CSan87a}*{pp.~158, (1.3.1)}).
It induces the \cmdm below with exact rows
\begin{displaymath}
  \xm@R15pt{
  0\ar[r] & H^1(\CO,\CQ_2)\ar[r]\ar[d] & H^2(\CO,\CM)\ar[r]\ar[d] & H^2(\CO,\CQ_1)\ar[d]\\
  0\ar[r] & H^1(\Omega,Q_2)\ar[r] & H^2(\Omega,M)\ar[r] & H^2(\Omega,Q_1)
  }
\end{displaymath}
where $Q_i=\CQ_i\times_{\CO}\Omega$ for $i=1,2$.
The left vertical arrow is injective by \cite{CSan87a}*{Theorem 4.1} and
the right one is injective by Shapiro's lemma and the injectivity for Brauer groups by \cite{MilneEC}*{IV, Corollary 2.6},
therefore the middle one is also injective.
\end{proof}

\begin{cor}\label{cor: Grothendieck--Serre for complexes over Hensel of DVR}
The \ttai $\Phi_v:\BBH^1(\CO_{U,v}^h,\CC)\to \BBH^1(K_v^h,C)$ induced by the inclusion $\CO^h_{U,v}\to K_v^h$ is injective for $v\in U^{(1)}$.
\end{cor}
\begin{proof}
Taking $\CO=\CO_{U,v}^h$, $\Omega=K_v^h$ and $\CJ=\CC$ in \cref{result: GS conj for complex of tori} yields the desired injectivity.
\end{proof}

To state a key step, we first construct a map $\bigoplus_{v\in X^{(1)}}\BBH^0(K_v^h,C)\to \BBH^1_c(U,\CC)$ for some non-empty open subset $U$ of $X_0$.
Take $\al\in \bigoplus_{v\in X^{(1)}}\BBH^0(K_v^h,C)$ supported outside some non-empty open subset $V$ of $U$, i.e. $\al_v=0$ for $v\in V$.
Applying \cref{result: l-e-s cpt-supp and 3 arrows lemma}(2) to $V$ sends $\al$ to $\BBH^1_c(V,\CC)$, and so $\al$ is sent to $\BBH^1_c(U,\CC)$ by the covariant functoriality of $\BBH^1_c(-,\CC)$.
The construction is independent of the choice of $V$ by the same argument as \cite{HSz16}*{pp.~11, (12)}.

\begin{prop}\label{result: key e-s degree 0}
There is an exact sequence
\[\tstopslim_{v\in X^{(1)}}\BBH^0(K_v^h,C)\to \BBH^1_c(U,\CC)\to \BBD_K^1(U,\CC)\to 0.\]
\end{prop}
\begin{proof}
The sequence is a complex by exactly the same argument of \cite{HSz16}*{Proposition 4.2}.
The surjectivity of the last arrow is just the definition of $\BBD_K^1(U,\CC)$.
Take $\al\in\ker\big(\BBH^1_c(U,\CC)\to \BBD_K^1(U,\CC)\big)$ and a non-empty open subset $V\subset U$.
Consider the diagram
\begin{equation*}
  \xm@R15pt{
  \BBH^1_c(V,\CC)\ar[r] & \BBH^1_c(U,\CC)\ar[r]\ar[d] & \bigoplus\limits_{v\in U\setminus V}\BBH^1(\kappa(v),i_v^*\CC)\ar[d]\\
  & \BBH^1(K,C)\ar[r] & \bigoplus\limits_{v\in U\setminus V}\BBH^1(K_v,C)
  }
\end{equation*}
where the right vertical arrow is constructed as the composite
\[
\BBH^1(\kappa(v),i_v^*\CC)\simeq \BBH^1(\CO_{U,v}^h,\CC)\to \BBH^1(K_v^h,C)\simeq\BBH^1(K_v,C).
\]
The first isomorphism is a consequence of
$H^1(\kappa(v),i_v^*\CT_i)\simeq H^1(\CO_{U,v}^h,\CT_i)$ for $i=1,2$
(see \cite{MilneADT}*{Chapter II, Proposition 1.1(b)})
and a d\'evissage argument.
The diagram commutes for the same reason as in the proof of \cite{HSz16}*{Proposition 4.2}.
Since the right vertical arrow is injective by \cref{cor: Grothendieck--Serre for complexes over Hensel of DVR},
a diagram chasing shows that $\al$ comes from $\BBH^1_c(V,\CC)$.
But $\al$ goes to zero in $H^1(K,C)$, we may take $V$ sufficiently small such that $\al$ already goes to zero in $\BBH^1(V,\CC)$ by Three Arrows Lemma,
i.e. $\al$ comes from $\bigoplus_{v\notin V}\BBH^0(K_v^h,C)$ by \cref{result: l-e-s cpt-supp and 3 arrows lemma}(2) and hence the desired sequence is indeed exact.
\end{proof}

Recall that $\BBD^1_K(U,\CC)$ is finite by \cref{result: D-1 and Sha-1 finite},
by \cref{lemma: on exactness of l-adic completion} there is an \ttes
\be\label{sequence: l-adic completion of etale cohomology}
\big(\tstopslim_{v\in X^{(1)}}\BBH^0(K_v^h,C)\big)^{(\ell)}\to \BBH^1_c(U,\CC)^{(\ell)}\to \BBD_K^1(U,\CC)^{(\ell)}\to 0.
\ee
Now we arrive at the global duality of the short complex $C$.

\begin{thm}\label{result: global duality degree 1 and 1}
There is a perfect, functorial in $C$, pairing of finite groups:
$$\Sha^1(C)\times \Sha^1(C')\to \QZ.$$
\end{thm}
\begin{proof}
We proceed by constructing a \ppair of finite groups $\Sha^1(C)\{\ell\}\times\Sha^1(C')\{\ell\}\to \QZ$ for a fixed prime $\ell$.
Define $\BBD^1_{\sh}(U,\CC)$ by the exact sequence
$0\to \BBD^1_{\sh}(U,\CC)\to \BBH^1(U,\CC)\to \prod_{v\in X^{(1)}}\BBH^1(K_v,C)$ for each $U\subset U_0$.
Dualizing the \ttes (\ref{sequence: l-adic completion of etale cohomology}) yields the following \cmdm (by a similar argument as \cite{CTH15}*{Proposition 4.3(f)}) with exact rows
\begin{equation*}
  \xm@R15pt{
  0\ar[r] & \BBD^1_{\sh}(U,\CC)\{\ell\}\ar[r]\ar@{-->}[d] & \BBH^1(U,\CC)\{\ell\}\ar[r]\ar[d] &
  \big(\prod\limits_{v\in X^{(1)}}\BBH^1(K_v,C)\big)\{\ell\}\ar[d]\\
  0\ar[r] & \big(\BBD_K^1(U,\CC')^{(\ell)}\big)^D\ar[r] & \big(\BBH^1_c(U,\CC')^{(\ell)}\big)^D\ar[r] & \big(\big(\bigoplus\limits_{v\in X^{(1)}}\BBH^0(K_v^h,C')\big)^{(\ell)}\big)^D.
  }
\end{equation*}
The first vertical arrow is induced by $\BBH^1(U,\CC)\{\ell\}\to\big(\BBH^1_c(U,\CC')^{(\ell)}\big)^D$ in view of
%the second one because of
the commutativity of the right square.
By local duality \cref{duality: local Hensel degree 0 1}, the right vertical arrow is an isomorphism,
and it follows that the kernels of the first two vertical arrows are identified.
Passing to the direct limit of the dashed arrow induces an exact sequence of abelian groups
\[
0\to \drl_U\drl_n\BBK_n(U)\to \drl_U\BBD^1_{\sh}(U,\CC)\{\ell\}\to \drl_U\big(\BBD_K^1(U,\CC')^{(\ell)}\big)^D.
\]
Recall that $\drl\BBK_n(U)$ is the divisible kernel of the middle vertical arrow introduced in \cref{duality: AV complex of tori degree 1 1}.
Note that the second limit is just $\Sha^1(C)\{\ell\}$ by definition of $\BBD^1_{\sh}(U,\CC)$.
In particular, the first limit is trivial being a divisible subgroup of a finite abelian group.
Now we can conclude the following isomorphisms of finite abelian groups
\[
 \BBD_K^1(U,\CC')^{(\ell)}=\BBD_K^1(U,\CC')\{\ell\}^{(\ell)}
=\Sha^1(C')\{\ell\}^{(\ell)}\simeq\Sha^1(C')\{\ell\}
\]
where the central equality holds by \cref{result: D-1 and Sha-1 finite}.
Summing up, there is an injection $\Sha^1(C)\{\ell\}\to \Sha^1(C')\{\ell\}^D$ and therefore the desired isomorphism $\Sha^1(C)\{\ell\}\simeq \Sha^1(C')\{\ell\}^D$ follows by exchanging the role of $C$ and $C'$.
\end{proof}

\section{Obstruction to weak approximation via special covering}
Let $G$ be a connected reductive group.
Let $G^{\ss}$ be the derived subgroup of $G$ and let $\Gsc\to G^{\ss}$ be the universal cover of $G^{\ss}$.
We consider the composition $\rho:\Gsc\to G^{\ss}\to G$.
Let $T\subset G$ be a \maxtorus over $K$ and let $\Tsc =\rho\ui(T)$.
Recall that $\Tsc $ is a \maxtorus of $\Gsc$.
We apply the above dualities to the morphism $\rho:\Tsc \to T$, i.e. to the complexes $C=[\Tsc \to T]$ and $C'=[T'\to (\Tsc)']$ concentrated in degree $-1$ and $0$.
%We first observe that the group $\Gsc$ satisfies \wa \wrt any finite set $S\subset X^{(1)}$ of places if it is quasi-split.
Recall $\Gsc$ satisfies $(*)$ if it has \wa and contains a quasi-trivial maximal torus.

\begin{prop}\label{result: quasi-split ss sc has WA over p-adic fct-field}
Let $H$ be a quasi-split semi-simple simply connected group over $K$.
Then $H$ satisfies \wa \wrt any finite set $S\subset X^{(1)}$ of places.
\end{prop}
\begin{proof}
Let $B$ be a Borel subgroup of $H$ defined over $K$ and
let $P$ be a maximal $K$-torus contained in $B$.
Applying \cite{HSz16}*{Lemma 6.7 and its proof} implies that $\wh{P}\simeq\wh{B}\simeq\pic(\ol{H/B})$ is a permutation module,
i.e. $P$ is a quasi-trivial torus.
Moreover, $P$ is a Levi subgroup of $B$ by \cite{BT65}*{Corollaire 3.14}.
Now \cite{Tha96}*{Corollary 1.5} yields a bijection
from the defect of \wa $\prod_{v\in S}H(K_v)/\ol{H(K)}_S$ for $H$
to that $\prod_{v\in S}P(K_v)/\ol{P(K)}_S$ for $P$ \wrt any finite set $S$ of places.
Here $\ol{H(K)}_S$ (resp. $\ol{P(K)}_S$) denotes the closure of $H(K)$ (resp. $P(K)$) in $\prod_{v\in S}H(K_v)$ (resp. $\prod_{v\in S}P(K_v)$) \wrt the product of $v$-adic topologies.
Since $P$ is a quasi-trivial torus,
by construction of Weil restriction it is an open subscheme of some affine space.
Hence $P$ satisfies \wa and so is $H$.
\end{proof}

\begin{prop}\label{result: quasi-split on red-ss-sc}
Let $H$ be a connected reductive group over $K$.
\TFAE $(1)$ $H$ is quasi-split; $(2)$ $\Hss$ is quasi-split; $(3)$ $\Hsc$ is quasi-split.
\end{prop}
\pff
\bitem
\itm
By \cite{SGA3III}*{Proposition 6.2.8(ii)},
there is a one-to-one correspondence between Borel subgroups of $H$ and that of $\Hss$,
so $H$ is quasi-split \ttiff $\Hss$ is quasi-split.

\itm
Suppose $\Hsc$ is quasi-split.
Since the universal covering $q:\Hsc\to \Hss$ is faithfully flat, \cite{Milne17}*{Proposition 17.68} implies that $q$ sends Borel subgroups of $\Hsc$ to Borel subgroups of $\Hss$.
\inpart $\Hss$ is quasi-split.

\itm
Suppose $\Hss$ is quasi-split and let $\Bss$ be a Borel subgroup of $\Hss$.
Then $\big(q\ui(\Bss)\big)\times_K\olK$ is a Borel subgroup of $\Hsc\times_K\olK$ by \cite{Milne17}*{Proposition 17.20},
i.e. $\Hsc$ is quasi-split.
\qed
\eitem

The following corollary says that our technical assumption $(*)$ holds if $G$ is quasi-split.

\begin{cor}\label{result: quasi-split groups have *}
If $H$ is a quasi-split connected reductive group over $K$,
then $\Hsc$ satisfies $(*)$.
\end{cor}
\begin{proof}
Indeed, $\Hsc$ is quasi-split by \cref{result: quasi-split on red-ss-sc} and
so it satisfies weak approximation by \cref{result: quasi-split ss sc has WA over p-adic fct-field}.
Moreover, $\Hsc$ contains a quasi-trivial maximal torus by \cite{HSz16}*{Lemma 6.7}.
\end{proof}

Suppose $\Gsc$ contains a quasi-trivial \maxtorus $\Tsc \subset \Gsc$.
Because $\Gsc\to G^{\ss}$ is an epimorphism,
the image $T^{\ss}$ of $\Tsc$ in $\Gss$ is again a maximal torus by \cite{Humphreys-alg-group}*{\S 21.3, Corollary C}.
Therefore we may choose a \maxtorus $T\subset G$ of $G$ such that $T\cap G^{\ss}=T^{\ss}$, i.e. $\Tsc =\rho\ui(T)$.
By \cite{Bor98}*{Section 2.4}, different choices of $[T^{\sconn}\to T]$ give rise to the same hypercohomology group and thus we are allowed to chose a quasi-trivial \maxtorus $\Tsc$.

\begin{thm}\label{result: OBS to WA qs-red groups}
Let $G$ be a connected reductive group such that $\Gsc$ satisfies $(*)$.
\benuma
\item Let $S\subset X^{(1)}$ be a finite set of places. There is an exact sequence
\be\label{sequence: OBS to WA red-group finite S}
  1\to \ol{G(K)}_S\to \tstprodlim_{v\in S}G(K_v)\to \Sha^1_S(C')^D\to \Sha^{1}(C)\to 1.
\ee
Here $\ol{G(K)}_S$ denotes the closure of the diagonal image of $G(K)$ in $\prod_{v\in S}G(K_v)$ for the product topology.

\item There is an exact sequence
\be\label{sequence: OBS to WA red-group all places}
  1\to \ol{G(K)}\to \tstprodlim_{v\in X^{(1)}}G(K_v)\to \Sha^1_{\og}(C')^D\to \Sha^{1}(C)\to 1.
\ee
Here $\ol{G(K)}$ denotes the closure of the diagonal image of $G(K)$ in $\prod_{v\in X^{(1)}}G(K_v)$ for the product topology.
\eenum
\end{thm}

\begin{eg}
Let us first look at two special cases of the sequence (\ref{sequence: OBS to WA red-group finite S}).\hfill
\benuma
\item
If $G$ is semi-simple, then there is an exact sequence $1\to F\to \Gsc\to G\to 1$ with $F$ finite and central in $\Gsc$. In particular, there are exact sequences
\[1\to F\to \Tsc \to T\to 1\qand 1\to F'\to T'\to (\Tsc)'\to 1.\]
Here $F'=\hom(F,\QZ(2))$ and the latter sequence is obtained from the dual isogeny of $\Tsc \to T$.
Consequently there are quasi-isomorphisms $C\simeq F[1]$ and $C'\simeq F'[1]$, and hence $\Sha^1_S(C)\simeq\Sha^2_S(F)$ and $\Sha^1(C')\simeq\Sha^2(F')$.
Therefore the exact sequence (\ref{sequence: OBS to WA red-group finite S}) reads as
\[1\to \ol{G(K)}_S\to \tstprodlim_{v\in S}G(K_v)\to \Sha^2_S(F')^D\to \Sha^{2}(F)\to 1.\]
Here the second arrow is given by the composite of the coboundary map $G(K_v)\to H^1(K_v,F)$ and the local duality $H^1(K_v,F)\times H^2(K_v,F')\to \QZ$,
and the last one is given by the global duality $\Sha^2(F)\times \Sha^2(F')\to \QZ$ for finite Galois modules
(see \cite{HSz16}*{(10) and Theorem 4.4} for details).

\item
If $G=T$ is a torus, then $C=[\Tsc \to T]$ is quasi-isomorphic to the complex $[0\to T]\simeq T$ and
its dual $C'=[T'\to (\Tsc)']$ is quasi-isomorphic to the complex $[T'\to 0]\simeq T'[1]$.
So we have $\Sha^1(C)\simeq\Sha^1(T)$ and $\Sha^1_S(C')\simeq\Sha^2_S(T')$.
Now the exact sequence (\ref{sequence: OBS to WA red-group finite S}) is of the following form
\[1\to \ol{T(K)}_S\to \tstprodlim_{v\in S}T(K_v)\to \Sha^2_S(T')^D\to \Sha^{1}(T)\to 1.\]
This is the obstruction to weak approximation for tori given by Harari, Scheiderer and Szamuely in \cite{HSSz15}.
\eenum
\end{eg}

The rest of this section is devoted to the proof of \cref{result: OBS to WA qs-red groups}.
For a field $L$ and $i=0,1$, we shall denote by $\ab^i:H^i(L,G)\to \BBH^i(L,C)$ the abelianization map in the sequel.
The construction of the abelianization map is set forth in \cite{Bor98}*{Section 3}.

\begin{lem}\label{result: ab-0 surj for qs-group}
Let $L$ be a field of characteristic zero and let $G$ be a connected reductive group over $L$ such that $\Gsc$ satisfies $(*)$.
Then the canonical map $\ab^0:H^0(L,G)\to \BBH^0(L,C)$ is surjective.
\end{lem}
\begin{proof}
Let $A\to B$ be a crossed module (see \cite{Bor98}*{Section 3}) of (not necessarily abelian) $\gal(\ol{L}|L)$-groups concentrated in degree $-1$ and $0$.
We write $\BBH_{\rel}^i(L,[A\to B])$ for the non-abelian \hchlg
in the sense of \cite{Bor98}*{Section 3} for $-1\le i\le 1$.
Now we view $\rho:\Tsc \to T$ and $\rho:\Gsc\to G$ as crossed modules of $\gal(\ol{L}|L)$-groups concentrated in degree $-1$ and $0$.
Recall $\Tsc $ is chosen to be quasi-trivial.
In particular, we have a \cmdm of crossed modules of $\gal(\ol{L}|L)$-groups
\begin{equation*}
  \xymatrix{
  [\Tsc \to T]\ar[r]\ar[d] & [\Tsc \to 1]\ar[d]\\
  [\Gsc\to G]\ar[r] & [\Gsc\to 1].
  }
\end{equation*}
Applying the functor $\BBH^0_{\rel}(\gal(\ol{L}|L),-)$ with values in the category of pointed sets and taking into account the identification $\BBH^0_{\rel}(\gal(\ol{L}|L),[A\to 1])\simeq H^1(\gal(\ol{L}|L),A)$ (see \cite{Bor98}*{Example 3.1.2(2)}), we obtain the following \cmdm of pointed sets
\begin{equation*}
  \xymatrix{
  \BBH^0_{\rel}(L,[\Tsc \to T])\ar[r]\ar[d] & H^1(L,\Tsc )=0\ar[d]\\
  \BBH^0_{\rel}(L,[\Gsc\to G])\ar[r] & H^1(L,\Gsc).
  }
\end{equation*}
By \cite{Bor98}*{Lemma 3.8.1}, there are isomorphisms of pointed sets
\[
\BBH^0(L,C)\simeq \BBH^0_{\rel}(L,[\Tsc \to T])\simeq \BBH^0_{\rel}(L,[\Gsc\to G])
\]
Since $\Tsc $ is quasi-trivial, we conclude that $\BBH^0(L,C)\to H^1(L,\Gsc)$ is the trivial map, that is, the canonical map $\ab^0:G(L)\to \BBH^0(L,C)$ is surjective (see \cite{Bor98}*{3.10}).
\end{proof}

We now proceed as in \cite{San81}. The first step is to show the following:

\begin{lem}\label{lemma: to reduce main result to special covering case}
Let $m\ge 1$ be an integer and let $Q$ be a quasi-trivial $K$-torus.
If the sequence \upshape{(\ref{sequence: OBS to WA red-group finite S})} is exact for $G^m\times_KQ$, then it is also exact for $G$.
\end{lem}
\begin{proof}
If the sequence (\ref{sequence: OBS to WA red-group finite S}) is exact for some finite direct product $G^m$, then  (\ref{sequence: OBS to WA red-group finite S}) is exact for $G$ as well.
We claim if (\ref{sequence: OBS to WA red-group finite S}) is exact for the product $G\times_KQ$ of $G$ by some quasi-trivial $K$-torus $Q$, then (\ref{sequence: OBS to WA red-group finite S}) is also exact for $G$.
Since $T\subset G$ is a \maxtorus of $G$, $T\times_KQ$ is a \maxtorus of $G\times_KQ$.
Moreover, $\s{D}(G\times_KQ)=G^{\ss}$, so we have a composite $\rho_Q:\Gsc\to G^{\ss}\to G\times_KQ$.
We introduce the complex $C_Q=[\Tsc \to T\times_KQ]$ which is concentrated in degree $-1$ and $0$.
Consider the following \cmdm
\begin{equation*}
  \xymatrix{
  H^1(K,\Tsc )\ar[r]\ar@{=}[d] & H^1(K,T\times_KQ)\ar[r]\ar[d]_{\simeq} & \BBH^1(K,C_Q)\ar[r]\ar[d] & H^2(K,\Tsc )\ar[r]\ar@{=}[d] & H^2(K,T\times_KQ)\ar[d]\\
  H^1(K,\Tsc )\ar[r] & H^1(K,T)\ar[r] & \BBH^1(K,C)\ar[r] & H^2(K,\Tsc )\ar[r] & H^2(K,T)
  }
\end{equation*}
where the second vertical map is an isomorphism since $Q$ is a quasi-trivial $K$-torus.
Applying \cite{HSSz15}*{Lemma 3.2(a)} yields isomorphisms $\Sha^1(C_Q)\simeq\Sha^1(T) \simeq \Sha^1(C)$.
Similarly, $\Sha^1_S(C_Q')^D\simeq \Sha^1_S(C')^D$ holds for any finite subset $S$ of places.
Finally, recall that quasi-trivial $K$-tori are $K$-rational, hence in particular $Q$ satisfies weak approximation.
It follows that the cokernel of the first map in (\ref{sequence: OBS to WA red-group finite S}) is stable under multiplying $G$ by a quasi-trivial torus and hence the exactness of (\ref{sequence: OBS to WA red-group finite S}) for $G\times_KQ$ yields the exactness of (\ref{sequence: OBS to WA red-group finite S}) for $G$.
Summing up, to prove the exactness of (\ref{sequence: OBS to WA red-group finite S}), we are free to replace $G$ by $G^m\times_KQ$ for some integer $m$ and some quasi-trivial $K$-torus $Q$.
\end{proof}

Suppose $\Gsc$ satisfies $(*)$.
By \cite{BT65}*{Proposition 2.2} and Ono's lemma \cite{San81}*{Lemme 1.7}, there exist an integer $m\ge 1$, quasi-trivial $K$-tori $Q$ and $Q_0$ such that $G^{\sconn,m}\times_KQ\to G^m\times_K Q_0$ is a central $K$-isogeny.
By \cref{lemma: to reduce main result to special covering case},
we may therefore assume that $G$ has a special covering $1\to F_0\to G_0\to G\to 1$ where $G_0$
satisfies \wabk and
has derived subgroup $\s{D}G_0=\Gsc$.
Moreover, $G_0$ contains a quasi-trivial \maxtorus $T_0$ over $K$
such that $T_0\cap \Gsc=\Tsc$ and that the sequence $1\to F_0\to T_0\to T\to 1$ is exact by construction.
Therefore we may assume $G$ admits a special covering in the sequel.

\vspace{1em}
The second step is to show the exactness at the first three terms:
\begin{lem}\label{lemma: main diagram chasing}
There is an exact sequence $1\to \ol{G(K)}_S\to \prod_{v\in S}G(K_v)\to \Sha^1_S(C')^D$.
\end{lem}
\proof
After passing to the dual isogeny of the exact sequence $1\to F_0\to T_0\to T\to 1$, we obtain $\Sha^2_S(F_0')\simeq\Sha^2_S(T')$ since $T_0'$ is quasi-trivial and $\Sha^2_{\og}(T_0')=0$ by \cite{HSSz15}*{Lemma 3.2(a)}.
Moreover, the \distri $T'\to (\Tsc)'\to C'\to T'[1]$ induces an isomorphism $\Sha^1_S(C')\simeq \Sha^2_S(T')$ for the same reason.
In particular, we obtain an isomorphism $\Sha^2_S(F_0')\simeq \Sha^1_S(C')$ which fits into the \cmdm with $F_0'=\hom(F_0,\QZ(2))$
\beac\label{diagram: main diagram chasing}
  \xm{
  G_0(K)\ar[r]\ar[d] & G(K)\ar[r]^-{\dif}\ar[d] & H^1(K,F_0)\ar[r]\ar[d] & H^1(K,G_0)\ar[d] \\
  \prod_{v\in S}G_0(K_v)\ar[r] & \prod_{v\in S}G(K_v)\ar[r]^-{\dif_v}\ar[d] & \prod_{v\in S}H^1(K_v,F_0)\ar[r]\ar[d] & \prod_{v\in S}H^1(K_v,G_0) \\
  & \Sha^1_S(C')^D\ar[r] & \Sha^2_S(F_0')^D
  }
\eeac
Recall the third column is exact by \cite{HSSz15}*{Lemma 3.1}.
We claim the coboundary map $\dif$ is surjective.
Since $F_0$ is finite and central in $G_0$, $F_0$ is contained in one, hence every, \maxtorus of $G_0$.
In particular, $F_0$ is contained in the quasi-trivial \maxtorus $T_0$ of $G_0$ and consequently the map $H^1(K,F_0)\to H^1(K,G_0)$ factors through $H^1(K,F_0)\to H^1(K,T_0)$.
By construction $T_0$ is quasi-trivial, so the vanishing $H^1(K,T_0)=0$ implies that $H^1(K,F_0)\to H^1(K,G_0)$ is the trivial map, i.e. $\dif$ is surjective.
Similarly, the coboundary map $\dif_v:G(K_v)\to H^1(K_v,F_0)$ is surjective for each place $v\in X^{(1)}$.
Since $G_0$ satisfies weak approximation by assumption, $G_0(K)$ has dense image in $\prod_{v\in S}G_0(K_v)$.
A diagram chasing now yields the desired exact sequence.\qed

\vspace{1em}
In order to prove \cref{result: OBS to WA qs-red groups}(1),
the last step is to show the exactness of the last three terms.
By definition there is an exact sequence
$$1\to \Sha^1(C')\to \Sha^1_{S}(C')\to \bigoplus_{v\in S}\BBH^1(K_v,C')$$
of discrete abelian groups.
Dualizing the sequence, we obtain an exact sequence of profinite groups:
\[
\tstprodlim_{v\in S}\BBH^{0}(K_v,C)^{\wedge}\to \Sha^{1}_S(C')^D\to \Sha^{1}(C)\to 1
\]
by \cref{duality: local cpx of tori degree 0 1} and \cref{result: global duality degree 1 and 1}.
Since $\Sha^1_S(C')\simeq \Sha^1_S(T')$ is a finite group, the groups $\prod_{v\in S}\BBH^0(K_v,C)$ and $\prod_{v\in S}\BBH^{0}(K_v,C)^{\wedge}$ have the same image in $\Sha^1_S(C')^D$.
By \cref{result: ab-0 surj for qs-group},
the canonical abelianization map $\prod_{v\in S}G(K_v)\to \prod_{v\in S}\BBH^0(K_v,C)$ is surjective
which guarantees the desired exactness.

\begin{proof}[Proof of \upshape{\cref{result: OBS to WA qs-red groups}(2)}]
Passing to the projective limit of (\ref{sequence: OBS to WA red-group finite S}) over all finite subset $S\subset X^{(1)}$ yields an exact sequence of groups
$$1\to \ol{G(K)}\to \tstprodlim_{v\in X^{(1)}}G(K_v)\to \Sha^1_{\og}(C')^D.$$
Dualizing the exact sequence of discrete groups
$$1\to \Sha^1(C')\to \Sha^1_{\og}(C')\to \tstopslim_{v\in X^{(1)}}\BBH^1(K_v,C')$$
yields an exact sequence of profinite groups
$$\tstprodlim_{v\in X^{(1)}}\BBH^0(K_v,C)^{\wedge}\to \Sha^1_{\og}(C')^D\to \Sha^1(C)\to 0.$$
Since $G(K_v)\to \BBH^0(K_v,C)$ is surjective
and $\prod_{v\in X\uun}\BBH^0(K_v,C)$ is dense in $\prod_{v\in X\uun}\BBH^0(K_v,C)^{\wedge}$,
it will be sufficient to show the image of $\prod_{v\in X^{(1)}}G(K_v)$ is closed in $\Sha^1_{\og}(C')^D$.
In view of diagram (\ref{diagram: main diagram chasing}), the quotient of $\prod_{v\in X^{(1)}}G(K_v)$ by $\ol{G(K)}$ is isomorphic to the quotient of the profinite group $\prod_{v\in X^{(1)}}H^1(K_v,F_0)$ by the closure of the image of $H^1(K,F_0)$.
Consequently, the quotient of $\prod_{v\in X^{(1)}}G(K_v)$ by $\ol{G(K)}$ is compact and hence the image of $\prod_{v\in X^{(1)}}G(K_v)$ in $\Sha^1_{\og}(C')^D$ is closed, as required.
\end{proof}

Actually the defect of weak approximation for $G$ can also be given by a simpler group $\Sha^2_{\og}(G^*)$ where $G^*$ is the group of multiplicative type whose character module is $\wh{G^*}=\pi_1^{\alg}(\ol{G})$.
Here $\pi_1^{\alg}$ (see \cite{Bor98}*{\S1} or \cite{CT08-resolution-flasque}*{\S6} for more details) denotes the algebraic fundamental group of a connected reductive group $G$
defined as $\pi_1^{\alg}(\ol{G})\ce\wc{T}/\rho_*\wc{\Tsc}$
where $\rho_*:\wc{\Tsc}\to \wc{T}$ is induced by $\rho:\Gsc\to G$.

\begin{prop}\label{remark: another defect of WA for reductive groups}
Let $G$ be a connected reductive group such that $\Gsc$ satisfies $(*)$.
Let $G^*$ be the group of multiplicative type whose character module is $\pi_1^{\alg}(\ol{G})$.
There is a surjective map of groups $\Sha^1_{\og}(G^*)\to \Sha^1_{\og}(C')$.
\end{prop}
\begin{proof}
Thanks to the \distri $T'\to (\Tsc)^{\prime}\to C'\to T'[1]$, there is a \cmdm
\begin{equation*}
  \xymatrix{
  0\ar[r] & \BBH^1(K,C')\ar[r] & H^2(K,T')\ar[r] & H^2(K,(\Tsc)^{\prime})\\
  & \Sha^1_{\og}(C')\ar[r]\ar[u] & \Sha^2_{\og}(T')\ar[r]\ar[u] & \Sha^2_{\og}((\Tsc)^{\prime}).\ar[u]
  }
\end{equation*}
with exact upper row.
%$$0\to \BBH^1(K,C')\to H^2(K,T')\to H^2(K,(\Tsc)^{\prime}).$$
Since $(\Tsc)^{\prime}$ is quasi-trivial, $\Sha^2_{\og}((\Tsc)^{\prime})=0$ by \cite{HSSz15}*{Lemma 3.2(a)} and it follows that $\Sha^1_{\og}(C')\simeq \Sha^2_{\og}(T')$ after a diagram chasing.

Let $1\to R\to H\to G\to 1$ be a flasque resolution of $G$ and consider the associated fundamental diagrams
(\ref{diagram: fundamental diagram associated to a flasque resolution: groups}) and
(\ref{diagram: fundamental diagram associated to a flasque resolution: maixmal tori}).
Since $\Tsc $ and $H^{\torus}$ are quasi-trivial tori,
we obtain $H^1(K,(T_H)')=0$ and $\Sha^2_{\og}((T_H)')=0$ by the associated long exact sequence of $1\to (H^{\torus})'\to (T_H)'\to (\Tsc)^{\prime}\to 1$.
Dualizing the middle row of the diagram (\ref{diagram: fundamental diagram associated to a flasque resolution: maixmal tori}) and considering the associated long exact sequence, we obtain
\be\label{sequence: iso between Sha groups}
  \Sha^1_{\og}(R')\simeq \Sha^2_{\og}(T')\simeq \Sha^1_{\og}(C').
\ee

Next, we relate $\Sha^2_{\og}(G^*)$ with $\Sha^2_{\og}(T')$.
Recall \cite{CT08-resolution-flasque}*{Proposition 6.8} that $\pi_1^{\alg}$ is an exact functor
from the category of connected reductive $K$-groups to that of $\gal(\ol{K}|K)$-modules of finite type.
Recall also that $\pi_1^{\alg}(\ol{R})$ of a $K$-torus $R$ is its module of cocharacters $\wc{R}$.
Thus there is a \cmdm
\beac\label{diagram: alg-pi-1}
\xm{
1\ar[r] & \wc{R}\ar[r]\ar@{=}[d] & \wc{T_H}\ar[r]\ar[d] & \wc{T}\ar[r]\ar[d] & 1 \\
1\ar[r] & \wc{R}\ar[r] & \pi_1^{\alg}(\ol{H})\ar[r] & \pi_1^{\alg}(\ol{G})\ar[r] & 1
}
\eeac
of Galois modules of finite type with exact lower row by the exactness of $\pi_1^{\alg}$.
Recall there is an anti-equivalence from the category of groups of multiplicative type to that of Galois modules of finite type which respects exact sequences, thus diagram (\ref{diagram: alg-pi-1}) corresponds to a \cmdm
\[\xm{
1\ar[r] & T'\ar[r] & T_H'\ar[r] & R'\ar[r] & 1 \\
1\ar[r] & G^*\ar[r]\ar[u] & H^*\ar[r]\ar[u] & R'\ar[r]\ar@{=}[u] & 1 \\
}\]
of groups of multiplicative type over $K$.
Taking Galois cohomology gives rise to \cmdms
\[\xm{
H^1(K,R')\ar[r]\ar@{=}[d] & H^2(K,T') & \Sha^1_{\og}(R')\ar[r]^-{\simeq}\ar@{=}[d] & \Sha^2_{\og}(T')\\
H^1(K,R')\ar[r] & H^2(K,G^*)\ar[u] & \Sha^1_{\og}(R')\ar[r]^-{\delta} & \Sha^2_{\og}(G^*).\ar[u]^-{\gamma}
%& \Sha^1_{\og}(R')\ar[r]\ar[ldd]\ar[d] & H^1(K,R')\ar[ldd]\ar[d]\\
%& \Sha^2_{\og}(G^*)\ar[r]\ar[ld] & H^2(K,G^*)\ar[ld]\\
%\Sha^2_{\og}(T')\ar[r] & H^2(K,T').
}\]
with $\Sha^1_{\og}(R')\simeq\Sha^2_{\og}(T')$ by (\ref{sequence: iso between Sha groups}).
It follows that $\gamma\circ\delta$ is an isomorphism,
so $\gamma$ is surjective.
Hence $\Sha^2_{\og}(G^*)\to\Sha^2_{\og}(T')\simeq\Sha^1_{\og}(C')$ is surjective.
\end{proof}

By \cite{HSz05}*{Appendix, Proposition (2)},
$\Sha^1_{\og}(C')^D\to \Sha^2_{\og}(G^*)^D$ is injective
($\Sha^1_{\og}(C')$ and $\Sha^2_{\og}(G^*)$ are discrete).
Thus there is an \ttes of groups by \cref{result: OBS to WA qs-red groups}(2)
\[1\to \ol{G(K)}\to \tstprodlim_{v\in X^{(1)}}G(K_v)\to \Sha^2_{\og}(G^*)^D.\]
Consequently, the defect of \wa may also be given by the group $\Sha^2_{\og}(G^*)$.

\section{Reciprocity obstruction to weak approximation}\label{section: reciprocity obstruction}
The next theorem is the promised generalization of \cite{HSSz15}*{Theorem 4.2} to the non-commutative case.
Let us briefly recall the construction of the pairing \cite{HSSz15}*{(17)} concerning unramified cohomology groups.
Let $Y$ be a smooth integral variety over $K$ with function field $K(Y)$.
Let $y_v:\e K_v\to Y$ be a $K_v$-point on $Y$.
Take $\al\in \hnr^3(K(Y),\munots{2})$ and lift it uniquely to $\al_{v}\in H^3(\CO_{Y,y_v},\munots{2})$
(here the uniqueness follows by the injective property for $\munots{j}$ over discrete valuation rings, see \cite{CT95SBB}*{\S3.6}).
Now $\al_v$ goes to $H^3(K_v,\munots{2})$ via $H^3(\CO_{Y,y_v},\munots{2})\to H^3(K_v,\munots{2})$.
Summing up, we obtain an evaluation pairing
\[Y(K_v)\times \hnr^3(K(Y),\munots{2})\to H^3(K_v,\munots{2}).\]
Taking the isomorphism $H^3(K_v,\QZ(2))\simeq\QZ$ for each $v\in X^{(1)}$ into account, we can construct a pairing
\be\label{pairing: unramified H 3}
\tstprodlim_{v\in X^{(1)}}Y(K_v)\times \hnr^3(K(Y),\QZ(2))\to \QZ
\ee

\begin{thm}\label{result: reciprocity OBS to WA for red-group}
Let $G$ be a connected reductive group such that $\Gsc$ satisfies $(*)$.
There exists a \ttai
\[
u:\Sha^1_{\og}(C')\to H^3_{\nr}(K(G),\QZ(2))
\]
such that each family $(g_v)\in\prod_{v\in X^{(1)}}G(K_v)$ satisfying $((g_v),\Im u)=0$ under the pairing
$(\ref{pairing: unramified H 3})$ lies in the closure $\ol{G(K)}$ with respect to the product topology.

More precisely,
the obstruction is given by
$\Im\big(H^3(G^c,\mu_n^{\ots2})\to \hnr^3(K(G),\QZ(2))\big)$
for some sufficiently large $n$.
\end{thm}
\begin{proof}
We first construct the \ttai $u:\Sha^1_{\og}(C')\to H^3_{\nr}(K(G),\QZ(2))$.
Let $\rho:\Gsc\to G$ be as before.
Let $\Tsc \subset \Gsc$ be a quasi-trivial \maxtorus and let $T\subset G$ be a \maxtorus such that $\Tsc =\rho\ui(T)$.
Recall that there is a fundamental diagram (\ref{diagram: fundamental diagram associated to a flasque resolution: maixmal tori}) associated with the flasque resolution $1\to R\to H\to G\to 1$, and recall also that $H^{\torus}$ is a quasi-trivial torus.
Moreover, there is a \ttai $\Sha^1_{\og}(C')\to H^1(K,R')$ via the inclusion $\Sha^1_{\og}(R')\to H^1(K,R')$ in view of (\ref{sequence: iso between Sha groups}).

Because $R$ is a flasque $K$-torus, applying \cite{CSan87a}*{Theorem 2.2(i)} implies that the class $[H]\in H^1(G,R)$ comes from a class $[Y]\in H^1(G^c,R)$,
where $G^c$ is a smooth compactification of $G$.
The pairing $R\ots^{\BL}R'\to \Z(2)[2]$ now induces a \ttai
$$H^1(K,R')\to H^4(G^c,\Z(2)),\ a\mt a_{G^c}\cup [Y]$$
with $a_{G^c}$ denoting the image of $a$ under $H^1(K,R')\to H^1(G^c,R')$.
The same argument as in \cite{HSSz15}*{Theorem 4.2} shows that there is a natural map
\[
H^4(G^c,\Z(2))\to H^3_{\nr}(K(G),\QZ(2))
\]
fitting into a \cmdm
\benn%\label{Gersten's-resolution-diagram}
  \xymatrix{
  H^4(G^c,\Z(2))\ar[r]\ar[d] & H^3_{\nr}(K(G),\QZ(2))\ar[d]\\
  H^4(K(G),\Z(2)) & H^3(K(G),\QZ(2)).\ar[l]_-{\simeq}
  }
\eenn

Now take an element $(g_v)\in \prod_{v\in X^{(1)}}G(K_v)$.
By \cref{result: OBS to WA qs-red groups}(2), $(g_v)$ lies in the closure $\ol{G(K)}$ if and only if $(g_v)$ is orthogonal to $\Sha^1_{\og}(C')$.
We consider the \cmdm (up to sign) of various cup-products
\begin{equation*}
  \xymatrix{
  \BBH^0(K_v,C) \ar@{}[r]|-{\bigtimes} & \BBH^1(K_v,C')\ar[d]\ar[r] & \QZ\ar@{=}[d]\\
  H^0(K_v,T)\ar[u]\ar[d]_{\delta_v} \ar@{}[r]|-{\bigtimes} & H^2(K_v,T')\ar[r] & \QZ\ar@{=}[d]\\
  H^1(K_v,R) \ar@{}[r]|-{\bigtimes} & H^1(K_v,R')\ar[r]\ar[u] & \QZ.
  }
\end{equation*}
Since $H^1(K_v,\Tsc )=0$ by the quasi-trivialness of $\Tsc $, $H^0(K_v,T)\to \BBH^0(K_v,C)$ is surjective.
In particular, there exists $t_v\in H^0(K_v,T)$ such that its image in $\BBH^0(K_v,C)$ equals $\ab^0_v(g_v)$.
The diagram together with Theorem \ref{result: OBS to WA qs-red groups} imply that $(g_v)\in \ol{G(K)}$ if and only if $(t_v)$ is orthogonal to $\Sha^1_{\og}(R')$.
Recall we have isomorphisms (\ref{sequence: iso between Sha groups}).
More explicitly, it means that
$$0=\tstsumlim_{v\in X^{(1)}}\lip a_v,\ab^0_v(g_v)\rip_v
=\tstsumlim_{v\in X^{(1)}}\ab^0_v(g_v)\cup a_v
=\tstsumlim_{v\in X^{(1)}}\delta_vt_v\cup a_v$$
for each $a\in \Sha^1_{\og}(C')\simeq\Sha^1_{\og}(R')$
(here the first two $a_v$ lie in $\BBH^1(K_v,C')$ while the last lies in $H^1(K_v,R')$).
Note that $\delta_vt_v$ is given by $t_v^*:H^1(T,R)\to H^1(K_v,R)$, $[T_H]\mt [T_H](t_v)=[Y](g_v)$.
Let $a_T$ be the image of $a\in H^1(K,R')$ in $H^2(K,T')$
and
let $a_{G^c}$ be the image of $a\in H^1(K,R')$ in $H^1(G^c,R')$.
It follows that
\be\label{sequence: iso of cup-product II}
 \tstsumlim_{v\in X^{(1)}}\delta_vt_v\cup a_v
=\tstsumlim_{v\in X^{(1)}}([T_H]\cup a_T)(t_v)
=\tstsumlim_{v\in X^{(1)}} ([Y]\cup a_{G^c})(g_v)
\ee
holds thanks to the \cmdm
\begin{equation*}
  \xymatrix{%@-1.3pc
  H^1(G^c,R)\ar[d] \ar@{}[r]|-{\bigtimes} & H^1(G^c,R')\ar[d]\ar[r] & \QZ\ar@{=}[d]\\
  H^1(G,R)\ar[d] \ar@{}[r]|-{\bigtimes} & H^1(G,R')\ar[d]\ar[r] & \QZ\ar@{=}[d]\\
  H^1(T,R) \ar@{}[r]|-{\bigtimes} & H^1(T,R')\ar[r] & \QZ.
  }
\end{equation*}
Note that the vanishing of the last term in (\ref{sequence: iso of cup-product II}) means that $(g_v)$ is orthogonal to the image of $u$ under the pairing $(-,-)$ which completes the proof.

Recall that $H^1(K,R')$ has finite exponent (say $N$).
Then $[Y]\in H^1(G^c,R)$ goes to $\dif_N([Y])\in H^2(G^c,{_N}R)$ induced a Kummer sequence,
and we can lift $a\in H^1(K,R')$ to $a_N\in H^1(K,{_N}R')$.
Now we obtain a class $(a_N)_{G^c}\cup\dif_N([Y])\in H^3(G^c,\mu_N^{\ots2})$
induced by ${_N}R\dtp {_N}R'\to \mu_N^{\ots2}$
and it restricts to a class in $\hnr^3(K(G),\mu_N^{\ots2})$.
\end{proof}

\begin{rmk}
The following argument was pointed out to the author by \JLCTbk.
Let $G$ be a quasi-split reductive group over $K$.
Let $B$ be a Borel subgroup of $G$ containing a maximal torus $T$ of $G$
and
let $B^-$ be the unique Borel subgroup of $G$ such that $B\cap B^-=T$.
Let $U^+=\rad^u(B)$ and $U^-=\rad^u(B^-)$ be respective unipotent radicals of $B$ and $B^-$.
Then the big cell of $G$ is $U^-\times T\times U^+$ by \cite{BCnotes-red}*{Proposition 1.4.11} (which is dense in $G$).
But $U^{\pm}$ are isomorphic to some affine spaces as varieties,
the big cell of $G$ is thus isomorphic to $T\times\A^N$ for suitable $N$,
i.e. $G$ is stably birational to its maximal torus $T$.
In this point of view, one sees that $\hnr^3(K(G),\QZ(2))\simeq\hnr^3(K(T),\QZ(2))$,
and that weak approximation for $G$ is equivalent to that for $T$.
\end{rmk}

\section{Appendix}
In this appendix, we compare cohomological obstructions to the Hasse principle and weak approximation for tori constructed in \cites{HSz16, HSSz15}, respectively.
Let $T$ be a $K$-torus and let $T'$ be its dual torus.
There is a canonical map constructed in \cite{HSz16}*{pp.~15, (20)}
\[H^2(K,T')\to H^3(T,\QZ(2))/H^3(K,\QZ(2))\]
via the \HSerre spectral sequence $H^p(K,H^q(\ol{T},\QZ(2)))\Ra H^{p+q}(T,\QZ(2))$.
Then we send $\Sha^2_{\og}(T')$ to $H^3(K(T),\QZ(2))/H^3(K,\QZ(2))$ through $H^3(T,\QZ(2))/H^3(K,\QZ(2))$.
Our goal is to show the image of $\Sha^2_{\og}(T')$ lies in $H_{\nr}^3(K(T),\QZ(2))/H^3(K,\QZ(2))$,
the unramified part of $H^3(K(T),\QZ(2))/H^3(K,\QZ(2))$.

Let $1\to R\to Q\to T\to 1$ be a \flasres with $R$ a flasque $K$-torus and $Q$ a quasi-trivial $K$-torus.

\begin{prop}\label{result: main in appendix}
There is a \cmdm $($up to sign$)$
\beac\label{diagram: main in appendix}
\xm{
H^1(K,R')\ar[r]\ar[d] & H^4(T,\Z(2))/H^4(K,\Z(2))\\
H^2(K,T')\ar[r] & H^3(T,\QZ(2))/H^3(K,\QZ(2))\ar[u].
}
\eeac
where the upper horizontal map is defined by the cup-product $H^1(K,R')\times H^1(T^c,R)\to H^4(T^c,\Z(2))$ $($see \cite{HSSz15}*{pp.~19} for details$)$, and the right vertical map is induced by the \ttes $0\to \Z(2)\to \Q(2)\to \QZ(2)\to 0$.
\end{prop}

\begin{cor}\label{result: Sha-2 of tori non-ram}
The image of $\Sha^2_{\og}(T')$ in $H^3(K(T),\QZ(2))/H^3(K,\QZ(2))$ lies in the unramified part.
\end{cor}
\begin{proof}
The map $H^1(K,R')\to H^4(T,\Z(2))$ factors through $H^4(T^c,\Z(2))\to H^4(T,\Z(2))$
by construction \cite{HSSz15}*{Theorem 4.2},
so the image of $\Sha^1_{\og}(R')\simeq \Sha^2_{\og}(T')$ lies in the unramified part $H^3_{\nr}(K(T),\QZ(2))/H^3(K,\QZ(2))$.
\end{proof}

The rest of the appendix is devoted to the proof of \cref{result: main in appendix}.
We begin with some observations on torsion groups under consideration.
Let $L|K$ be a finite Galois extension splits both $R$ and $T$.
The vanishing of $H^1(T_L,R_L)=0$ implies that the class $[Q]\in H^1(T,R)$ is torsion by a \rescores argument.
The spectral sequence $H^p(K,\ext^q_{\ol{K}}(\ol{R}',\ol{T}'))\Ra \ext^{p+q}_K(R',T')$ together with the vanishing $\ext^1_{\ol{K}}(\ol{R}',\ol{T}')=0$ implies that
$H^1(K,\hom_{\ol{K}}(\ol{R}',\ol{T}'))\to \ext^1_K(R',T')$ is an isomorphism.
Thus the group $\ext^1_K(R',T')$ is torsion.
We choose a suitable integer $n$ such that the classes $[Q]\in H^1(T,R)$ and $[Q']\in\ext^1_K(R',T')$ are both $n$-torsion.

\vspace{1em}
\noindent\textbf{Step 1: We verify the commutativity of diagram (\ref{diagram: main in appendix}) with a different construction of the left vertical arrow as in diagram (\ref{diagram: cup-product 2 times 5}).}
\vspace{1em}

The \Kumseq $1\to {_n}R\to R\to R\to 1$ yields a surjection $H^1(T,{_n}R)\to {_n}H^1(T,R)$, i.e. $[Q]=\iota_n([Q_n])$ for some class $[Q_n]\in H^1(T,{_n}R)$ with $\iota_n:H^1(T,{_n}R)\to H^1(T,R)$ induced by the Kummer sequence.

Let $p:T\to \e K$ be the structural morphism.
Let $\BD(K)$ be the derived category of bounded complexes of Galois modules.
We consider the object $ND(T)=(\tau_{\le 1}\BR p_*\munots{2})[1]$
in $\BD(K)$ which fits into a \distri
\be\label{triangle: for ND(T)}
\munots{2}[1]\to ND(T)\to H^1(\ol{T},\munots{2})\to\munots{2}[2].
\ee
We will follow \cite{HSk13descent}*{Proposition 1.1} to construct a map
\[\chi:H^1(T,{_n}R)\to \hom_K({_n}R',ND(T)).\]
The pairing ${_n}R\deots{_n}R'\to \munots{2}$ yields a map $H^1(T,{_n}R)\to H^1(T,\ulhom({_n}R',\munots{2}))$.
Moreover, we obtain a map $H^1(T,\ulhom({_n}R',\munots{2}))\to \ext_T^1({_n}R',\munots{2})$
from the exact sequence in low degrees associated to the local-to-global spectral sequence
\[H^p(T,\ulext_T^q({_n}R',\munots{2}))\Rightarrow \ext_T^{p+q}({_n}R',\munots{2}).\]
Since $\BR\hom_T({_n}S',-)=\BR\hom_K({_n}S',-)\circ \BR p_*(-)$ is a composition, formally there is a canonical isomorphism
\[\ext_T^1({_n}R',\munots{2})\simeq R^1\hom_K({_n}R',\BR p_*\munots{2}).\]
Because $\tau_{\ge 2}\BR p_*\munots{2}$ is acyclic in degrees $0$ and $1$, we obtain an isomorphism
\[R^1\hom_K({_n}R',\tau_{\le1}\BR p_*\munots{2})\simeq R^1\hom_K({_n}R',\BR p_*\munots{2})\]
from the distinguished triangle
$\tau_{\le1}\BR p_*\munots{2}\to \BR p_*\munots{2}\to \tau_{\ge2}\BR p_*\munots{2}\to ND(T)$.
Now $\chi$ is just the composition
\[H^1(T,{_n}R)\to \ext^1_T({_n}R',\mu_n^{\ots 2})\simeq R^1\hom_K({_n}R',\tau_{\le1}\BR p_*\munots{2})=\hom_K({_n}R',ND(T)).\]
All the above constructions yield a diagram of cup-products
\beac\label{diagram: cup-product 4 times 3}
  \xm{
  H^1(K,R')\ar[d]_{\dif_n} \ar@{}[r]|-{\bigtimes} & H^1(T,R)\ar[r] & H^4(T,\Z(2))\\
  H^2(K,{_n}R')\ar@{=}[d] \ar@{}[r]|-{\bigtimes} & H^1(T,{_n}R)\ar[d]\ar[u]_{\iota_n}\ar[r] & H^3(T,\munots{2})\ar@{=}[d]\ar[u]_{\dif}\\
  H^2(K,{_n}R')\ar@{=}[d] \ar@{}[r]|-{\bigtimes} & \hom_K({_n}R',\BR p_*\munots{2}[1])\ar@{=}[d]\ar[r] & H^2(K,\BR p_*\munots{2}[1])\\
  H^2(K,{_n}R') \ar@{}[r]|-{\bigtimes} & \hom_K({_n}R',ND(T))\ar[r] & H^2(K,ND(T))\ar[u]
  }
\eeac
where the upper diagram commutes by functoriality of the cup product pairing
(see \cite{HSz16}*{diagram (26)} for more details),
and the lower two diagrams commute by \cite{MilneEC}*{Proposition V.1.20}.
Diagram (\ref{diagram: cup-product 4 times 3}) gives the commutativity of the left two squares of the following diagram
(where the second square comes from the lower three rows):
\beac\label{diagram: cup-product 2 times 5}
\xm{
H^1(K,R')\ar[r]\ar[d] & H^2(K,{_n}R')\ar[r]\ar[d] & H^2(K,ND(T))\ar[d]\ar[r] & H^2(K,{_n}T')\ar[r]\ar[d] & H^2(K,T')\ar[d]\\
H^4(T,\Z(2)) & H^3(T,\munots{2})\ar[l]\ar@{=}[r] & H^3(T,\munots{2})\ar[r] & \frac{H^3(T,\munots{2})}{H^3(K,\munots{2})}\ar[r] & \frac{H^3(T,\QZ(2))}{H^3(K,\QZ(2))}.
}
\eeac
The right two squares in diagram (\ref{diagram: cup-product 2 times 5}) commute by construction of the \HSerre spectral sequences.
Passing to the quotient by respective subgroup of constants and taking limits over all $n$ imply the commutativity of diagram (\ref{diagram: main in appendix}).
Consequently, we are done if the upper row of diagram (\ref{diagram: cup-product 2 times 5}) gives the coboundary map $H^1(K,R')\to H^2(K,T')$ induced by the \ses $1\to T'\to Q'\to R'\to 1$ of tori.

\vspace{1em}
\noindent\textbf{Step 2: We check that composite of arrows in the upper row of diagram (\ref{diagram: cup-product 2 times 5}) is just the desired coboundary map in diagram (\ref{diagram: main in appendix}).}
\vspace{1em}

The \Kumseq $1\to {_n}T'\to T'\to T'\to 1$ induces a surjection
$\ext^1_K(R',{_n}T')\to {_n}\ext^1_K(R',T')$, so the class $[Q']$ lifts to a class $[M_n]\in \ext^1_K(R',{_n}T')$.
Similarly, the \Kumseq $1\to {_n}R'\to R'\to R'\to 1$ induces an isomorphism
$\hom_K({_n}R',{_n}T')\to \ext^1_K(R',{_n}T')$
by the vanishing of $\hom_K(R',{_n}T')=0$.
Hence there is a \cmdm
\beac\label{diagram: dt of tori 3 times 3}
\xm{
  0\ar[r] & {_n}R'\ar[r]\ar[d] & R'\ar[r]\ar[d] & R'\ar[r]\ar@{=}[d] & 0\\
  0\ar[r] & {_n}T'\ar[r]\ar[d] & M_n\ar[r]\ar[d] & R'\ar[r]\ar@{=}[d] & 0\\
  0\ar[r] & T'\ar[r] & Q'\ar[r] & R'\ar[r] & 0.
}
\eeac
Applying the functor $H^n(K,-)$ to diagram (\ref{diagram: dt of tori 3 times 3}) yields a \cmdm
\[\xm{
H^1(K,R')\ar[r]\ar[d] & H^2(K,{_n}R')\ar[d]\\
H^2(K,T') & H^2(K,{_n}T')\ar[l]
}\]
which tells us the composite $H^1(K,R')\to H^2(K,{_n}R')\to H^2(K,{_n}T')\to H^2(K,T')$ is exact the coboundary map $H^1(K,R')\to H^2(K,T')$ induced by the bottom row of diagram (\ref{diagram: dt of tori 3 times 3}).

It remains to show the map $H^2(K,{_n}R')\to H^2(K,{_n}T')$ obtained from $\ext^1_K(R',{_n}T')\simeq\hom_K({_n}R',{_n}T')$ coincides with the composition $H^2(K,{_n}R')\to H^2(K,ND(T))\to H^2(K,{_n}T')$.
The cup-product pairing
\[\Phi(-,-):H^1(\ol{T},{_n}R)\times H^0(\ol{K},{_n}R')\to H^1(\ol{T},\munots{2}),\]
defines a map
$H^1(\ol{T},{_n}R)\to \hom_{\ol{K}}({_n}R'(\ol{K}),H^1(\ol{T},\munots{2}))$.
Again there is a \cmdm by \cite{MilneEC}*{Proposition V.1.20}:
\begin{equation*}
  \xymatrix{
  H^1(\ol{T},{_n}R) \ar[d]\ar@{}[r]|-{\bigtimes} & {_n}R'(\ol{K})\ar@{=}[d]\ar[r] & H^1(\ol{T} ,\munots{2})\ar@{=}[d]\\
   \hom_{\ol{K}}({_n}R'(\ol{K}),\BR \ol{p}_*\munots{2}[1]) \ar@{}[r]|-{\bigtimes} & {_n}R'(\ol{K})\ar@{=}[d]\ar[r] & H^0(\ol{K},\BR \ol{p}_*\munots{2}[1])\\
  \hom_{\ol{K}}({_n}R'(\ol{K}),ND({\ol{T}})) \ar[u]^{\simeq}\ar@{}[r]|-{\bigtimes} & {_n}R'(\ol{K})\ar[r] & H^0(\ol{K},ND(\ol{T}))\ar[u].
  }
\end{equation*}
which may be rewritten into the following \cmdm
\[\xm{
  H^1(\ol{T},{_n}R)\ar[r]\ar[rd]^-{\chi}\ar[d]_{\Phi(-,-)}
  & \hom_{\ol{K}}({_n}R'(\ol{K}),\BR \ol{p}_*\munots{2}[1])\\
  \hom_{\ol{K}}({_n}R'(\ol{K}),H^1(\ol{T},\munots{2}))
  & \hom_{\ol{K}}({_n}R'(\ol{K}),ND(\ol{T}))\ar[l]^-{\al}\ar[u]_{\simeq}\\
}\]
The arrow $\al$ is induced by the \distri (\ref{triangle: for ND(T)}).
Now $\al\circ\chi=\Phi(-,-)$ says that ${_n}R'(\ol{K})\to ND(\ol{T})\to H^1(\ol{T},\munots{2})$
is the same as $\Phi([Q_n],-)$.
In particular, $H^2(K,{_n}R')\to H^2(K,ND(T))\to H^2(K,{_n}T')$ is the same as $H^2(K,{_n}R')\to H^2(K,{_n}T')$ obtained from the identification $\ext^1_K(T,{_n}R)\simeq\hom_K({_n}R',{_n}T')$.

Univerit\'e Paris-Sud, Institut de Math\'ematique d'Orsay, B\^atiment 307, 91405 Orsay, France\\
\indent E-mail address: yisheng.tian@u-psud.fr

\end{document}